\documentclass[hidelinks,12pt]{amsproc}
\pdfoutput=1 
\usepackage{tikz}
\usetikzlibrary{calc,3d}
\usepackage{xurl}
\allowdisplaybreaks
\usepackage[english]{babel}
\usepackage[pdftex,paper=a4paper,portrait=true,textwidth=450pt,textheight=675pt,tmargin=3cm,marginratio=1:1]{geometry}
\usepackage{amsfonts}
\usepackage{csquotes}
\usepackage{graphicx}
\usepackage{caption}
\usepackage{subcaption}
\usepackage{amsmath}
\usepackage{amsthm}
\usepackage{amssymb}
\usepackage{bbm}
\usepackage{cancel}
\usepackage{comment}
\usepackage{color}
\usepackage{biblatex}
\usepackage[all,cmtip]{xy}
\usepackage{graphicx}
\usepackage{mathabx}
\usepackage{tikz-cd}
\usepackage{epigraph} 
\usepackage{hyperref}
\usepackage{stmaryrd}
\usepackage{enumitem}
\usepackage{mathtools}
\usepackage[mathscr]{eucal}
\mathtoolsset{showonlyrefs}
\newtheorem{theorem}{Theorem}[section]

\newtheorem{proposition}[theorem]{Proposition}
\newtheorem{lemma}[theorem]{Lemma}

\theoremstyle{definition}
\newtheorem{definition}[theorem]{Definition}
\newtheorem{example}[theorem]{Example}
\newtheorem{remark}[theorem]{Remark}

\DeclareFontFamily{OT1}{rsfs}{}
\DeclareFontShape{OT1}{rsfs}{n}{it}{<-> rsfs10}{}
\DeclareMathAlphabet{\curly}{OT1}{rsfs}{n}{it}

\newcommand\C{\mathbb C}

\newcommand\Q{\mathbb Q}

\newcommand\Z{\mathbb Z}

\newcommand\Coh{\mathrm{Coh}}

\newcommand\vir{\mathrm{vir}}

\newcommand\rk{\operatorname{rk}}

\newcommand\Hom{\operatorname{Hom}}

\newcommand\Ext{\operatorname{Ext}}

\newcommand\INTO{\ar@{^{(}->}[r]}

\DeclareRobustCommand{\SkipTocEntry}[4]{}

\def\Pic{\mathrm{Pic}}
\def\GL{\mathrm{GL}}
\def\PGL{\mathrm{PGL}}

\def\and{\quad\mathrm{and}\quad}

\DeclareMathOperator{\Ch}{ch}
\DeclareMathOperator{\Br}{Br}
\DeclareMathOperator{\End}{End}

\newcommand{\dirk}[1]{{\color{blue}{[#1]}}}
\newcommand{\ww}[1]{{\color{orange}{[#1]}}}
\newcommand{\ch} [3][2]{\operatorname{ch}_{#2}( #3 )}

\begin{document}

\title{Euler characteristics of moduli of twisted sheaves on Enriques surfaces}

\author{Dirk van Bree \and Weisheng Wang}

\begin{abstract}
Let $Y$ be an Enriques surface and let $\mathcal{A}$ be an Azumaya algebra corresponding to the non-trivial Brauer class. Let $M$ be the moduli space of stable twisted sheaves on Enriques surfaces with twisted Chern character $M^H_{\mathcal{A}/Y}(2,c_1,\Ch_2)$ with virtual dimension $N$. We show that the virtual Euler characteristic $e^\vir(M)$ only depends on $N$, more precisely, $e^\vir(M)=0$ when $N$ is odd and $e^\vir(M)=2\cdot e(Y^{[\frac{N}{2}]})$ when $N$ is even.
\end{abstract}

\maketitle

\section{Introduction}

\subsection{Main result}
The objective of enumerative geometry is the computation of enumerative invariants of moduli spaces, most notably their (virtual) Euler characteristics. Among the moduli spaces studied are the moduli spaces of sheaves and that of twisted sheaves. The goal of the current paper is to compute the virtual Euler characteristics of moduli spaces of twisted sheaves on Enriques surfaces. 

Recall that an Enriques surface is a smooth projective surface $Y$ with nontrivial canonical bundle $\omega_Y$ satisfying $\omega_Y^2 = \mathcal{O}_Y$ and $H^1(Y, \mathcal{O}_Y) = 0$. Many other definitions exist~\cite{Dolgachev, CossecDolgachevLiedtke}. A fundamental property of an Enriques surface is that its fundamental group is $\mathbb{Z}/2\mathbb{Z}$ and that its universal cover is a K3 surface $\tilde{Y}$ with a fix-point free involution $\iota$.

The notion of twist is provided by the \emph{Brauer group} $\Br(Y)$, which always admits a map to the \emph{topological Brauer group} $H^3(Y, \mathbb{Z})$. In the case of an Enriques surface, there are isomorphisms $\Br(Y) \cong H^3(Y, \mathbb{Z}) = \mathbb{Z}/2\mathbb{Z}$. Thus, there is a unique nontrivial Brauer class $B_Y$. A twisted sheaf is an $\mathcal{A}$-module, where $\mathcal{A}$ is an Azumaya algebra representing the Brauer class. One can then define the moduli space of stable twisted sheaves, along with its virtual structure. Denote the Hilbert scheme of $n$ points on $Y$ by $Y^{[n]}$. Our main result is then the following.

\begin{theorem}
\label{maintheorem}
    Let $Y$ be an Enriques surface and let $\mathcal{A}$ be an Azumaya algebra of degree $2$ on $Y$ that represents the unique nontrivial Brauer class. Fix a polarization\footnote{All rank $2$ twisted sheaves on $Y$ are automatically stable, therefore we don't need polarization and the g.c.d. condition $\gcd(2,c_1H)=1$. The moduli space can be denoted by $M_{\mathcal{A}/Y}(\Ch_\mathcal{A})$. We write $H$ here fore latter convenience when pass to moduli of ordinary sheaves where polarization will indeed be needed.  } $H$ of $Y$. Let $c_1\in H^2(Y,\Z)_{\mathrm{tf}}$ be a class modulo torsion. Let $\Ch_\mathcal{A}=(2,c_1,\Ch_2)$ be a twisted Chern character. Let $M = M^H_{\mathcal{A}/Y}(\Ch_\mathcal{A})$ be the moduli space of stable twisted sheaves with twisted Chern character $\Ch_\mathcal{A}$. Denote by $N$ the virtual dimension of $M$. Then we have an equality
    \begin{equation}
        \label{eq:main-result-intro}
        e^{\vir}(M) = \begin{cases}
            0 & \text{ if $N$ is odd,}\\
            2\cdot e(Y^{[\frac{N}{2}]}) & \text{ if $N$ is even,}
        \end{cases}
    \end{equation}
    where $e^\vir(\cdot)$ denotes the virtual Euler characteristic.
\end{theorem}

Up until this moment, the enumerative study of moduli of twisted sheaves has been on the \emph{topologically trivial} case, which means that the image of the Brauer class in $H^3(X, \mathbb{Z})$ vanishes~\cite{HotchkissPerry, BGJK, Jiang2022, JiangKool}, in particular in the study of S-duality. This assumption leads to the availability of technical tools, most notably the B-field (see e.g.~\cite{HuybrechtsStellari}). In this paper, we break new ground by doing computations in the case where the Brauer class does not vanish in $H^3(X, \mathbb{Z})$.

\subsection{Background} 
Moduli of sheaves on Enriques surfaces have always been related to moduli of sheaves on the K3 cover, since they were studied by H. Kim~\cite{Kim_1998}. Other references for the untwisted theory are Nuer and Yoshioka~\cite{Nuer2015, Nuer2020, Yoshioka2018}. Recently, Oberdieck has worked on the enumerative geometry of Enriques surfaces. His work encompasses Gromov-Witten, Vafa-Witten and Pandharipande-Thomas invariants, among others~\cite{oberdieck2023curvecountingenriquessurface, OberdieckRefinedI, OberdieckRefinedII}. He obtained closed formulas for the Gromov-Witten invariants and Vafa-Witten invariants. We would like to mention that among his results is the same formula~\eqref{eq:main-result-intro}, but then for untwisted moduli spaces including both higher rank cases and the cases where there are strictly semi stable sheaves. 

Moduli of twisted sheaves were first considered by C\v{a}ld\v{a}raru~\cite{Caldararu2000}, Yoshioka~\cite{Yoshioka2006}, Lieblich~\cite{Lieblich2007} and Hoffmann-Stuhler~\cite{hoffmann_2005_moduli}. Recently, Reede~\cite{reede2019ranksheavesquaternionalgebras, reede2024enriquessurfacestrivialbrauer} has worked on moduli of twisted sheaves on Enriques surfaces specifically. He showed that the results of Kim hold in the twisted setting quite generally. We will use his work, and earlier work by Beauville~\cite{beauville2009brauergroupenriquessurfaces} on the Brauer group of an Enriques surface.

The enumerative geometry of twisted sheaves is important for the S-duality conjecture of Vafa and Witten~\cite{Vafa1994} as formulated in~\cite{Jiang2022, JiangKool, vBThesis}. Witten~\cite{witten_1998_adscft} modified the conjecture for surfaces with $H^1(Y, \mathbb{Z}) \neq 0$, such as Enriques surfaces, where $H^1(Y, \mathbb{Z}) \cong \mathbb{Z}/2\mathbb{Z}$. The conjecture involves the Euler characteristic of the moduli space of twisted sheaves with fixed \emph{determinant}. This requirement is more refined than our result, since for any line bundle $L$, $c_1(L)$ and $c_1(L \otimes \omega_Y)$ are the same modulo torsion. If the rank is odd, then the two moduli spaces with fixed determinant $L$ and $L\otimes \omega_Y$ are isomorphic. It is shown in \cite{beckmann_2020_birational} that for a primitive
Mukai vector of odd rank $v$, we have $M^H_Y(v)$ is birational to $Y^{[N/2]}$ with $N$ being the dimension of $M^H_Y(v)$. Therefore, in this case we have  $e(M^H_Y(v))=e(Y^{[N/2]})$. 

Recall also that the S-duality conjecture exists for each rank $r$. Such a conjecture involves sheaves and twisted sheaves of rank $r$, and Brauer classes whose order divides $r$. Currently, S-duality is studied mostly when $r$ is prime. However, for an Enriques surface, the conjecture only involves nontrivial Brauer classes when $r = 2$ (assuming $r$ is prime). Therefore, we believe that rank two twisted sheaves are the most significant case to study.

\subsection{Idea of proof and plan of the paper}

Our proof of Theorem~\ref{maintheorem} has two main ideas. The first is to deform our Enriques surface $Y$ to a more suitable one, where the moduli space of twisted sheaves is smooth and the Brauer class $B_Y$ of the Enriques surface vanishes on the K3 cover $\tilde{Y}$. Work of Reede~\cite{reede2024enriquessurfacestrivialbrauer} shows that the moduli of twisted sheaves can be realised as the fixed point locus of $\iota^*$ acting on the moduli of twisted sheaves $\tilde{Y}$ (recall that $\iota$ is the fixed-point-free action on $\tilde{Y}$). 

On $\tilde{Y}$, the moduli spaces of twisted sheaves and untwisted sheaves are isomorphic and there is also an action of $\iota^*$ on the space of untwisted sheaves, the fixed point locus of which is the moduli of untwisted sheaves. However, the isomorphism and the actions of $\iota^*$ are incompatible. For this reason, we seek a solution to this problem that is less geometric. Thus, the second main idea is to use the Lefschetz fixed point theorem to compute the Euler characteristic. This approach is inspired by~\cite[App. C]{oberdieck2023curvecountingenriquessurface}. In this approach, we only need to find an isomorphism of cohomology that respects the action of $\iota^*$, not of the moduli spaces themselves.

The desired isomorphism is provided by Markman operators~\cite{markman2005monodromy, Oberdieck_2022, wang2024virasoroconstraintsk3surfaces}. In fact, we can identify the cohomology of the moduli of twisted sheaves on $\tilde{Y}$ with the Hilbert scheme of points on $\tilde{Y}$, which essentially gives us the Euler characteristic of the Hilbert scheme on $Y$. The factor 2 in the theorem stems from the fact that we get the Euler characteristic of the moduli of rank one sheaves with $c_1 = 0$ modulo torsion, which means both determinant $\mathcal{O}_Y$ and $\omega_Y$, giving us two copies of the Hilbert scheme.

We will first recall the theory of moduli of twisted sheaves in Section~\ref{sec:twisted-sheaves}. We will specialise to Enriques surfaces in Section~\ref{sec:twisted-enriques} and study the deformation theory of Enriques surfaces. We then introduce Markman operators in Section~\ref{sec:markman-operators} before doing the final computation in Section~\ref{sec:top-eul-char}.

\subsection{Future directions}
Much of the structure of moduli spaces of twisted sheaves on Enriques surfaces is still unknown. The following questions arose in the preparation of this paper. We already mentioned the problem of fixing the determinant instead of the $c_1$ modulo torsion when discussing S-duality. We mention two problems.
\begin{enumerate}
    \item Is the moduli space of twisted sheaves always smooth of the expected dimension?
    \item If the moduli space of twisted sheaves has odd virtual dimension, is it necessarily empty?
\end{enumerate}
Reede~\cite{reede2024enriquessurfacestrivialbrauer} showed that in the moduli spaces of Enriques surfaces, there is an infinite union of divisors where the answer to both questions is ``yes''. In fact, this union is the Beauville locus, Def.~\ref{def:beauville-locus}. However, we are not aware of examples where pathological behaviour occurs.

Of course, one can also ask whether it is possible to compute more refined invariants, such as integrals of polynomials in descendent classes. Currently, our methods do not allow this. While our deformation technique still works, the Lefschetz fixed point formula computes only the Euler characteristic. It would be interesting to relate descendent classes on the Enriques surface to descendent classes on the K3 cover.

\subsection{Notation and conventions}

We will always work over the complex numbers. Furthermore, all cohomology groups will be the singular cohomology groups in the complex analytic topology, unless otherwise stated.

\subsection{Acknowledgments}
The authors would like to thank their supervisor Martijn Kool for many useful suggestions. The authors express their gratitude to Woonam Lim, Fabian Reede and Georg Oberdieck for useful conversations.
D.vB. is supported by NWO grant VI.Vidi.192.012.
W.W. is supported by the ERC Consolidator Grant FourSurf 101087365.

\section{Twisted sheaves and invariants} \label{sec:twisted-sheaves}

In this section $X$ will be a smooth projective surface over the complex numbers. Later, we will specialise $X$ to be an Enriques surfaces $Y$, and we will always work over the complex numbers.

\begin{definition}
    An \emph{Azumaya algebra} $\mathcal{A}$ on $X$ is a sheaf of algebras that locally (in the \'etale topology) is isomorphic to a matrix algebra $M_r(\mathcal{O}_X)$. The integer $r$ is called the \emph{degree} of $\mathcal{A}$.
\end{definition}

\begin{example}
\label{ex:trivial-algebra}
    Trivially, for any vector bundle $E$, $\End(E^\vee)$ is an Azumaya algebra. Indeed, the required condition already holds Zariski locally. More is known: the vector bundle $E$ is uniquely determined up to tensoring with a line bundle on $X$. In other words, any other such $E$ is of the form $E \otimes L$. The Azumaya algebras of this form are called \emph{trivial} because these algebras are precisely those whose image in the Brauer group (see below) is zero.
\end{example}

\begin{definition}
\label{def:brauer-group}
    The Brauer group $\Br(X)$ of a smooth projective surface is defined as $H^2_{\text{\'et}}(X, \mathbb{G}_m)$. Here, the cohomology group is with respect to the \'etale topology.
\end{definition}

Any Azumaya algebra $\mathcal{A}$ on $X$ induces an element of $\Br(X)$. The most concrete way to see this is the following. By the Skolem-Noether theorem, the set of isomorphism classes of Azumaya algebras of degree $r$ is the same as the set of isomorphism classes of $\PGL_r$-bundles \cite{vBThesis}. The short exact sequence $0 \to \mathbb{G}_m \to \GL_r \to \PGL_r \to 0$ then induces a connecting homomorphism $H^1(X, \PGL_r) \to H^2_{\text{\'et}}(X, \mathbb{G}_m) = \Br(X)$ which associates to each $\PGL_r$-bundle an Azumaya algebra.

This particular map has been the topic of much study. In~\cite{De_Jong_result_of_gabber} it is shown that this map is surjective as soon as $X$ is a quasi-projective variety of any dimension. However, for surfaces, a much better result is known, which is De Jong's Period-Index theorem for surfaces~\cite{Jong2004}.

\begin{theorem}
\label{thm:de-jong-period-index}
    For any smooth projective surface $X$ and $\alpha \in \mathrm{Br}(X)$ of order $r$, there is a degree $r$ Azumaya algebra of class $\alpha$.
\end{theorem}

From now on, we fix an Azumaya algebra $\mathcal{A}$. By \cite{Jong2004}, we can assume that our Azumaya algebra is unobstructed, i.e., if $Y$ is contained in any smooth family, then $\mathcal{A}$ extends \'etale locally around $Y$. More precisely: the Lemma 3.2 in \cite{Jong2004} allows one to extend the Azumaya algebra along an embedding whose ideal sheaf has square zero. Then, as the Lemma 6.1 in \cite{Jong2004} shows, by the Grothendieck's existence theorem and the Artin's approximation theorem one can extend the Azumaya algebra to an \'etale neighborhood.

In this paper, we define twisted sheaves using Azumaya algebras. Other models are also available, notably Yoshioka's $Y$-sheaves on a Brauer-Severi variety $Y$ over $X$~\cite{Yoshioka2006} and Lieblich's moduli of twisted sheaves on a gerbe $\mathcal{G}$ over $X$~\cite{Lieblich2007}. For an extensive comparison between these models, see~\cite{vBThesis}.

\begin{definition}
    A \emph{twisted sheaf} is a (left) $\mathcal{A}$-module. Generally, we will only consider coherent $\mathcal{A}$-modules.
\end{definition}

\begin{example}
    \label{ex:trivial-algebra-modules}
    Consider the trivial algebra $\mathcal{A} = \End(E^\vee)$. Then every $\mathcal{A}$-module $M$ is of the form $M = F \otimes E^\vee$. In other words, the assignment $F \mapsto F \otimes E^\vee$ is an equivalence from the category of coherent $\mathcal{O}_X$-modules to that of coherent twisted sheaves. Note that the equivalence is not unique: one can replace $E$ by $E \otimes L$, which changes the equivalence by a line bundle. This equivalence is called Morita equivalence, one can refer to Chapiter 3 in \cite{vBThesis} for more details. 
\end{example}

Much of the moduli theory of twisted sheaves can be developed in a similar way to the moduli theory of ordinary sheaves. An extra obstacle is that twisted sheaves do not admit a natural notion of global sections. We therefore introduce two different notions of the Hilbert polynomial of twisted sheaves: one defined using Hirzebruch-Riemann-Roch and one using the underlying $\mathcal{O}_X$-module of a twisted sheaf.

\begin{definition}
    The twisted Chern character of a twisted sheaf $E$ is defined as $\Ch_{\mathcal{A}}(E) = \Ch(E) / \sqrt{\Ch(\mathcal{A})}\in H^\ast(X,\Q)$. The associated twsited Mukai vector is $v_{\mathcal{A}}(E) = \Ch_{\mathcal{A}}(E)\sqrt{\mathrm{Td}_X}$. Fix a polarisation $H$ of $X$. The twisted Hilbert polynomial $P(E)$ is defined as \[
    P_{\mathcal{A}}(E)(n) = \int_X \Ch_{\mathcal{A}}(E) e^{c_1(H)n} \mathrm{Td}_X.
    \]
    Finally, we denote by $P(E)$ the Hilbert polynomial of the underlying $\mathcal{O}_X$-module $E$.
\end{definition}

We can now define $\mu$-stability and Gieseker stability in the usual way. However, since we have two different Hilbert polynomials, we have two different sets of notions of stability. Under the usual gcd condition, these all coincide\footnote{In~\cite{vBThesis} it is stated that the notions of stability always coincide, which is false. However, the gcd condition is assumed there in all important results.}.

\begin{remark}
\label{rk2onlystable}
    Later, we will consider an Azumaya algebra of degree two that induces a nontrivial class in the Brauer group. In this case, the stability condition for twisted sheaves of rank two is automatic, since twisted sheaves of rank 1 do not exist.
\end{remark}

Before we continue, we have to introduce the notion of determinants of twisted sheaves and some other supporting notions.

\begin{definition}
    Let $M$ be a twisted sheaf. We will denote the degree $i$ part of $\Ch_{\mathcal{A}}(M)$ by $\Ch_{\mathcal{A},i}(M)$. Then the \emph{rank} of $M$ is $\rk_\mathcal{A}(M):=\Ch_{\mathcal{A},0}(M)$. Alternatively, it is $\rk_\mathcal{A}(M) = \rk M / \sqrt{\rk \mathcal{A}}$ where $\rk(\bullet)$ means the rank of $\mathcal{O}_X$-module. $\rk_\mathcal{A}(M)$ is always an integer.
\end{definition}

\begin{proposition}
\label{prp:twisted-determinant}
    Let $M$ be a twisted sheaf of rank $k$. And Let $\mathcal{A}$ be a Azumaya algebra on $X$ of degree $r$.
    \begin{enumerate}
        \item There exists a rank 1 locally projective $\mathcal{A}^{\otimes k}$-module $\det_\mathcal{A} M$.
        \item If $L$ is any line bundle on $X$, then $\det_\mathcal{A}(M \otimes L) = \det_\mathcal{A} M \otimes L^{\otimes k}$.
        \item If $k$ is a multiple of $r$, then we can define a map \[\widehat{\det}:\Coh(X,\mathcal{A})\to \Pic(X),\] such that:\\
        \quad (i) $\det_\mathcal{A}(M)=\widehat{\det}(M)\otimes \det_\mathcal{A} \mathcal{A}^{\oplus d}$ \\
        \quad (ii)  $c_1(\widehat{\det}(M))=\Ch_{\mathcal{A},1}(M)$.
    \end{enumerate}
\end{proposition}

\begin{proof}
    Suppose that $\mathcal{A}$ is a trivial algebra $\End(E^\vee)$ (see Example~\ref{ex:trivial-algebra}). Then $M$ is of the form $F \otimes E^\vee$ (Example~\ref{ex:trivial-algebra-modules}). We define $\det M = \det F \otimes (E^\vee)^k$. Note that this is well-defined: it is invariant under replacing $E$ by $E \otimes L$. This is why we have to use $(E^\vee)^{\otimes k}$. It is immediate to verify the second property from this definition. 

    Let us assume $k=dr$. Let $p:\mathfrak{X}\to X$ be the $\mathbb{G}_m$-gerbe $\alpha(X)$ defined in \cite{vBThesis}. Then we have $\Tilde{\mathcal{A}}:=p^\ast \mathcal{A}=\End(E^\vee)$, where $E$ is a rank $r$ twisted sheaf, i.e. $E$ has weight $1$ in the decomposition 
    \[Q\Coh(\mathfrak{X}) = \prod _{i\in \Z} Q\Coh^i(\mathfrak{X}).\]
    We also have: $\Tilde{M}:=p^\ast M= F \otimes E^\vee$ where $F$ is a rank $k$ sheaf on $\mathfrak{X}$ of weight $1$. We define $\widehat{\det}$ as:
    \[\widehat{\det}(M):=p_\ast\left(\det (F)\otimes (\det (E^{\oplus d}))^{-1}\right),\] where $\det$ is the ordinary determinant on $\mathfrak{X}$. Notice that the weight of $\det F$ is $k$, the weight of $(\det E^{\oplus d})^{-1}$ is $-k$ and the tensor of those two has weight $0$, therefore the pushforward by $p_\ast$ gives an ordinary sheaf on $X$.

    Notice that $\det_\mathcal{A}\mathcal{A}^{\oplus d}=p_\ast\left( \det(E^{\oplus d})\otimes (E^\vee)^{\otimes k} \right)$, (i) is obvious.
    On the level of Chow ring we have $A(\mathfrak{X})_\Q\cong A(X)_\Q$ and the inverse of $p^\ast$ is $rp_\ast$, using this fact (ii) is obtained by direct calculation.
\end{proof}

\begin{proposition}
\label{MoritaEq}
    Let $X$ be a surface with polarisation $H$. Let $\Ch_{\mathcal{A}} = r + c_1 + \Ch_2$ be a cohomology class on $X$ satisfying $\gcd(r, c_1.H) = 1$ and $r$ is a multiple of $\deg \mathcal{A}$. Then:
    \begin{enumerate}
        \item For twisted sheaves with twisted Chern character $\Ch_{\mathcal{A}}$, the notions of $\mu$-(semi)-stability, Gieseker-(semi)-stability with respect to the twisted Hilbert polynomial and ordinary Hilbert polynomial all agree.
        \item If $\mathcal{A} = \End(E^\vee)$ represents the trivial class in $\Br(X)$ then the equivalence $F \mapsto F \otimes E^\vee$ preserves the stability condition.
    \end{enumerate}
\end{proposition}

\begin{proof}
    In this case, we have that for a twisted sheaf $E$, we have $\frac{\Ch_{\mathcal{A},1}(E)}{r} = \frac{\Ch_1(E)}{\rk E}$, so the notion of $\mu$ does not depend on the definition of Hilbert polynomial that we consider. This usual argument then implies that $\mu$-stability and $\mu$-semi-stability coincide when we take the ordinary Hilbert polynomial in our definition. But then all the notions agree.

    For the second claim we compute that
    \begin{equation}
    \label{eq:chern-relation}
    \Ch_{\mathcal{A}}(F \otimes E^\vee) = \frac{\Ch(F \otimes E^\vee)}{\sqrt{\Ch(\End(E^\vee))}} = \Ch(F) \cdot e^{-c_1(E)/r}.
    \end{equation}
    This change of Chern character does not affect $\mu$-stability, the proof is the same as the proof that $\mu$-stability is invariant under tensoring by a line bundle.
\end{proof}

\begin{remark}
    Equation~\eqref{eq:chern-relation} uses the fact that $\dim X \leq 2$. Our definition of $\Ch_{\mathcal{A}}$ needs to be modified if we consider varieties of higher dimension, see~\cite{vBThesis} for a treatment using gerbes.

    We also remark that the second part of this proposition also works if $\mathcal{B}$ is any Azumaya algebra, and $\mathcal{A} = \mathcal{B} \otimes \End(E)$ is a Brauer-equivalent Azumaya algebra. In this way, we obtain that in general, the notion of stability does not depend on the choice of Azumaya algebra (assuming the gcd condition).
\end{remark}

Having introduced stability, we now discuss the moduli theory of twisted sheaves. For the absolute moduli space of generically simple twisted sheaves one can refer \cite{hoffmann_2005_moduli}. In the following, we will need the existence of relative moduli spaces, so we will discuss the general version from~\cite{vBThesis}.

\begin{theorem}
\label{thm:moduli-space-existence}
    Let $\mathcal{X} \to S$ be a family of smooth projective surfaces and fix a relative polarisation $\mathcal{H}$. Assume also that $\mathcal{A}$ is an Azumaya algebra on $\mathcal{X}$. Let $r$ be an integer that is a multiple of $\deg \mathcal{A}$, let $\mathcal{L}$ be a line bundle on $\mathcal{X}$ such that $gcd(r,c_1(\mathcal{L}|_s)\cdot \mathcal{H}|_s)=1\quad \forall s\in S$ and let $n \in \mathbb{Q}$. Then there exists a projective morphism $M^{\mathcal{H}}_{\mathcal{A}/\mathcal{X}/S}(r, \mathcal{L}, n) \to S$ which is the relative moduli space of stable twisted sheaves and whose fiber over $s \in S$ is the absolute moduli space of twisted sheaves of rank $r$ and with determinant $\widehat{\det}(E)=\mathcal{L}$ and $\Ch_{\mathcal{A},2}(E)=n$ where $E$ is a twisted sheaf in this absolute moduli space.
\end{theorem}

For this theorem we need an Azumaya algebra $\mathcal{A}$ on $\mathcal{X}$. However, recall that we noted that any Azumaya algebra on a single fibre could be assumed to be unobstructed \cite{Jong2004}. This has the consequence that we can fix an Azumaya algebra on a single fibre and then extend it to an \'etale neighbourhood of the corresponding point in $S$.

We also note that these moduli spaces are coarse in general.

\begin{theorem}[Theorem 3.2.83 in \cite{vBThesis}]
\label{thm:perfect-obs-th}
    In the setting of Theorem~\ref{thm:moduli-space-existence} assume $S$ is smooth and $\mathcal{X}\to S$ is smooth. Then there is a perfect relative perfect obstruction on $\mathcal{M}^{\mathcal{H}}_{\mathcal{A}/\mathcal{X}/S}(r, \mathcal{L}, n)$ (this is the moduli stack as a $\mu_l$-gerbe over $M^\mathcal{H}_{\mathcal{A}/\mathcal{X}/S}(r, \mathcal{L}, n)$ for some $l$) given by:
    \[R\Hom_\pi(\mathcal{E},\mathcal{E})_0^\vee[-1]\to \mathbb{L}_{\mathcal{M}^{\mathcal{H}}_{\mathcal{A}/\mathcal{X}/S}(r, \mathcal{L}, n)},\]
    where $\pi$ is the projection to the moduli stack and $\mathcal{E}$ is the universal sheaf. The subscript zero denotes the trace-free part. This perfect obstruction theory descends to $M^\mathcal{H}_{\mathcal{A}/\mathcal{X}/S}(r, \mathcal{L}, n)$ and commutes with base change. 
    
    In particular, let $E\in M^\mathcal{H}_{\mathcal{A}/\mathcal{X}/S}(r, \mathcal{L}, n)$ be a closed point on the fiber $X$ above $s\in S$ then the tangent space is given by $\Ext_X^1(E,E)_0$ and the obstruction space by $\Ext_X^2(E,E)_0$.
     
    Therefore, the moduli space $M^\mathcal{H}_{\mathcal{A}/\mathcal{X}/S}(r, \mathcal{L}, n)$ admits a virtual class, it commutes with regular local immersions and flat morphisms. 
\end{theorem}

This theorem gives us access to the virtual Euler characteristics $e^{\vir}(M^{\mathcal{H}|_X}_{\mathcal{A}/X}(r, \mathcal{L}|_X, n))$  and the theorem also tells us that they are deformation invariant. For the virtual Euler characteristic, one can refer \cite{ionuciocanfontanine_2009_virtual,fantechi_2010_riemannroch}. Here, the notation indicates the absolute moduli space over a surface.

\begin{remark}
    Instead of Azumaya algebras to formulate the theory of twisted sheaves, one can use $\mu_r$-gerbes. Every Azumaya algebra of degree $r$ induces such a gerbe and the gerbe encodes more-or-less the notion of twisted sheaves together with the notion of the Chern character and stability. Gerbes, while technical, can be used to avoid deep results about the existence of Azumaya algebras. Also, constructions such as the determinant of twisted sheaves are much more natural in this language. Furthermore, this version of the theory generalises better to higher dimensions, as Theorem \ref{thm:de-jong-period-index} is false in higher dimensions. See~\cite{vBThesis,Lieblich2007} for details on this topic.
\end{remark}

\section{Twisted sheaves on Enriques surfaces}
\label{sec:twisted-enriques}

From now on we let $Y = X$ be an Enriques surface. We will study the Brauer group and moduli space of twisted sheaves, and relate it to the moduli space on the K3 cover of $Y$. This cover is denoted by $\pi : \tilde{Y} \to Y$ and the involution on $\tilde{Y}$ is denoted $\iota$.

\subsection{Brauer group of Enriques surfaces and the Beauville locus}

The Brauer group of an Enriques surface is known, see for example~\cite{beauville2009brauergroupenriquessurfaces}.

\begin{proposition}
    Let $Y$ be an Enriques surface. Then we have an isomorphism $H^2_{\text{\'et}}(Y, \mathbb{G}_m) = \mathrm{Br}(Y) \to H^3(Y, \mathbb{Z}) \cong \mathbb{Z}/2\mathbb{Z}$.
\end{proposition}

\begin{definition}
    Let $Y$ be an Enriques surface. We denote by $B_Y$ the unique nonzero element of $\mathrm{Br}(Y)$. We also let $\mathcal{A}$ be any unobstructed Azumaya algebra of degree two that induces $B_Y$.
\end{definition}

Note that $\mathcal{A}$ exists by Theorem~\ref{thm:de-jong-period-index} and the comments thereafter. Note that the fact that $B_S$ is nonzero implies that there is no \emph{B-field}, i.e., a class $\xi \in H^2(Y, \mathbb{Z})$ that maps to $B_Y$ under the composition
\[
H^2(Y, \mathbb{Z}) \to H^2(Y, \mathbb{Z}/r\mathbb{Z}) \cong H^2_{\text{\'et}}(Y, \mu_r) \to \Br(Y).
\]
Indeed, any class in this image maps to zero in $H^3(X, \mathbb{Z})$. See the introduction for the relevance of the nonexistence of such a $\xi$.

We will now study the moduli space of Enriques surfaces in order to show that we can deform to a more suitable surface. The suitable Enriques surfaces lie in a locus that we will call the \emph{Beauville locus}, after Beauville~\cite{beauville2009brauergroupenriquessurfaces}, who proved its existence. On the Beauville locus, we have that $\pi^*B_Y = 0$, where $\pi : \tilde{Y} \to Y$ is the K3 cover of $Y$. More explicitly, the following is Beauville's result.

\begin{theorem}
    \label{thm:Beauville-locus}
    If $Y$ is an Enriques surface, let $\pi : \tilde{Y} \to Y$ be its K3 cover and let $\iota : \tilde{Y} \to \tilde{Y}$ be the involution. Let $M$ be the (coarse) moduli space of Enriques surfaces.
    \begin{enumerate}
        \item For a fixed $n < -1$, the set of Enriques surfaces $Y$ for which there exists a line bundle $L$ on $\tilde{Y}$ with $c_1(L)^2 = 4n+2$ and $\iota^*L = L^{-1}$ is a hypersurface in $M$, which we denote by $H_n$.
        \item The locus of Enriques surfaces $Y$ for which $\pi^*B_Y = 0 \in \Br(\tilde{Y})$ is the infinite union $\bigcup_{n < -1} H_n$ of distinct hypersurfaces.
    \end{enumerate}
\end{theorem}

If $Y$ is such that $\pi^*B_Y = 0$, then $\pi^*\mathcal{A}$ is a trivial algebra, and it turns out that it can be written as $\End(\mathcal{O}_{\tilde{Y}} \oplus L)$. This is one way of constructing $L$ (see~\cite{reede2024enriquessurfacestrivialbrauer}).

\begin{definition}
    \label{def:beauville-locus} We let $B = \bigcup_{n < -1} H_n$ be the \emph{Beauville locus}, the locus of Enriques surfaces for which $B_Y$ vanishes on the K3 cover. For any integer $N < -1$, we also denote by $B_N$ the locus $\bigcup_{n \leq N} H_n$.
\end{definition}

\begin{lemma}
\label{lem:deform-line-bundle}
    Let $\mathcal{Y} \to S$ be a smooth family of Enriques surfaces. Let $L$ be a line bundle defined on $\mathcal{Y}_s$ for some $s \in S$. Then $L$ extends to an \'etale neighbourhood of $s$.
\end{lemma}
\begin{proof}
    This is a standard application of deformation theory. The obstruction space to deforming $L$ is $\Ext^2(L, L) = H^2(Y, \mathcal{O}_Y)=0$. The result now follows using Grothendieck's existence theorem and Artin's approximation theorem.
\end{proof}

\begin{theorem}
\label{thm:deform-to-beauville-locus}
    Fix an integer $N$ and let $Y$ be an Enriques surface with a fixed polarisation $H$ and equipped with a line bundle $L$. Then there exists a smooth projective family $\mathcal{Y} \to S$ with $S$ connected and smooth and a relative polarisation $\mathcal{H}$ and a line bundle $\mathcal{L}$ such that there exists a point $p \in S$ for which $(\mathcal{Y}|_s, \mathcal{H}|_s, \mathcal{L}|_s) = (Y, H, L)$ and there exists a point $q \in S$ for which $\mathcal{Y}|_q \in B_N$. In fact, we can assume that the set of such $q$ in dense in $S$.
\end{theorem}

\begin{proof}
    Let $\mathcal{M}$ be the moduli stack of Enriques surfaces. By ~\cite[Sec. 4]{Dolgachev}, $\mathcal{M}$ is an irreducible smooth unirational Artin stack of dimension 10. And we have a surjection to the coarse moduli space $\mathcal{M} \to M$. Now note that $B_N$, as an infinite union of hypersurfaces, is dense in $M$. Therefore, its preimage is dense in $\mathcal{M}$. By definition of Artin stack, there is a union of schemes $\bigsqcup_\alpha S_\alpha$ and a smooth surjective morphism $\bigsqcup_\alpha S_\alpha\to \mathcal{M}$. Now, let $p$ be the closed point in $\mathcal{M}$ representing $Y$. Let $S$ be a component of $\bigsqcup_\alpha S_\alpha$ that contains $p$, notice that $S$ is smooth since $\mathcal{M}$ is smooth. This gives us a smooth family with smooth base. By shrinking $S$ further, $L$ and $H$ also extend by Lemma~\ref{lem:deform-line-bundle}. By~\cite[Thm. 1.2.27]{Lazarsfeld2004}, we also find that $H$ remains ample in a neighbourhood of $Y$. Finally, we note that $B_N$ is dense, and shrinking our space does not affect this notion.
\end{proof}

\subsection{The moduli space of twisted sheaves on Enriques surfaces}

In Theorem~\ref{thm:moduli-space-existence}, we recalled the existence of moduli spaces of twisted sheaves with fixed determinant. However, we will be fixing a Mukai vector $v$ instead of a determinant. In the following, $v$ will always denote a Mukai vector of rank 2. The next proposition is just a corollary of Thm.~\ref{thm:moduli-space-existence} and Thm.~\ref{thm:perfect-obs-th}.

\begin{remark}
    Twisted sheaves in general do not admit integral Chern classes, only a rational Chern character or Mukai vector. Nevertheless, in our situation, we consider rank two sheaves over a degree two Azumaya algebra. In this case, the determinant of such a sheaf is a proper line bundle, by Prop.~\ref{prp:twisted-determinant}. We can also formally define a second Chern class using the formula $c_2 = \frac{1}{2}c_1^2 - \Ch_2$. This will not give integral results in general. If $Y$ is an Enriques surface in the Beauville locus, then one can show that $c_2$ defined as above is of the form $\frac{1}{2}\mathbb{Z} + \frac{1}{4}$. To show this, use that $\pi^*\mathcal{A} = \End(\mathcal{O}_{\tilde{Y}} \oplus L)$, that $L^2 \equiv 2 \mod 4$. In the following text, we will also use the fact $c_1(L)c_1(\pi^*M) = 0$ for all line bundles $M$ on $Y$ implied by $\iota^*L = L^{-1}$.
\end{remark}

To study the relative moduli space of stable sheaves over a family, we need a way to fix Chern character over a family. We will recall some facts for this purpose.
Let us recall the definition of Hodge bundle and its filtration. let $\mathcal{Y}\to S$ be a smooth family with $S$ smooth and connected. The Hodge bundle is defined as:
\[\mathcal{H}_{D R}^k(\mathcal{Y} / S):=R^kf_*\left(\Omega_{\mathcal{Y} / S}^\bullet\right),\]
where \[\Omega_{\mathcal{Y} / S}^\bullet:=\left[\mathcal{O}_{\mathcal{Y}} \xrightarrow{d_{D R}} \Omega_{\mathcal{Y} / S}^1 \xrightarrow{d_{D R}} \Omega_{\mathcal{Y} / S}^2 \right]\] is the \textit{algebraic De Rham complex}. Using the stupid truncations 
\[\Omega_{\mathcal{Y} / S}^{\geqslant p}:=\left[0 \rightarrow \cdots \rightarrow 0 \rightarrow \Omega_{\mathcal{Y} / S}^p \xrightarrow{d_{D R}} \Omega_{\mathcal{Y} / S}^{p+1} \xrightarrow{d_{D R}} \cdots \xrightarrow{d_{D R}} \Omega_{\mathcal{Y} / S}^2\right].\] One can define the a filtration of the Hodge bundle $\mathcal{H}_{D R}^{p+q}(\mathcal{Y} / S)$:
\[F^p \mathcal{H}_{D R}^{p+q}(\mathcal{Y} / S):=\operatorname{im}\left(R^{p+q} f_*\left(\Omega_{\mathcal{Y} / S}^{\geqslant p}\right) \rightarrow R^{p+q} f_*\left(\Omega_{\mathcal{Y} / S}^{\bullet}\right)\right).\]
The filtration has the following property: $\frac{F^p \mathcal{H}_{D R}^{p+q}(\mathcal{Y} / S)}{F^{p+1} \mathcal{H}_{D R}^{p+q}(\mathcal{Y} / S)} \cong R^q f_* \Omega_{\mathcal{Y} / S}^p$. The Hodge bundle, the filtration and the quotient are all vector bundles and stable under base change. Now, we consider a section horizontal section $\tilde{v}_1\in \Gamma(S,F^1 \mathcal{H}_{D R}^{2}(\mathcal{Y} / S))$ with respect to the Gauss-Manin connection on the Hodge bundle. $\tilde{v}_1$ is a locally constant family of cohomology classes $\tilde{v}_{1,b}\in H^{1,1}(\mathcal{Y}_b,\C)$ on the fibers $\mathcal{Y}_b:=\mathcal{Y}\times_S \{b\}$. Therefore, to fix a Mukai vector over the family $\mathcal{Y}\to S$, we need to fix a horizontal section $\tilde{v}_1\in \Gamma(S,F^1 \mathcal{H}_{D R}^{2}(\mathcal{Y} / S))$ such that locally on fibers we have $\tilde{v}_{1,b}\in H^{1,1}(\mathcal{Y}_b,\C)\cap H^{2}(\mathcal{Y}_b,\Z)$. Similarly, we can fix horizontal sections $\tilde{\mathbbm{1}}\in \Gamma(S,F^0 \mathcal{H}_{D R}^{0}(\mathcal{Y} / S))$ and $\tilde{\boldsymbol{\operatorname{p}}}\in \Gamma(S,F^2 \mathcal{H}_{D R}^{4}(\mathcal{Y} / S))$ such that locally on fibers give respectively the fundamental class and the point class.

\begin{proposition}
\label{prp:existence-mod-space-enriques}
    Let $f : \mathcal{Y} \to S$ be a smooth family of Enriques surfaces with $S$ smooth. Let $\mathcal{H}$ be a relative polarisation. Fix an above-described horizontal section $\tilde{v}_1\in \Gamma(S,F^1 \mathcal{H}_{D R}^{2}(\mathcal{Y} / S))$ such that $\tilde{v}_{1,b}.\mathcal{H}|_b$ is odd $\forall b\in S$. Then, \'etale locally on $S$, there is an Azumaya algebra $\mathcal{A}$ of degree two on $\mathcal{Y}$ that represents the nontrivial Brauer class in each fibre. There also exists a relative projective moduli space $M^H_{\mathcal{A}/\mathcal{Y}/S}(\widetilde{v_\mathcal{A}}) \to S$ of stable twisted sheaves with twisted Mukai vector $\widetilde{v_\mathcal{A}}$, where $\widetilde{v_\mathcal{A}}=(v_0\cdot\tilde{\mathbbm{1}},\tilde{v}_1,v_2\cdot\tilde{\boldsymbol{\operatorname{p}}})$ with $v_0\in \Z,v_2\in \Q$ and $\tilde{\mathbbm{1}},\tilde{\boldsymbol{\operatorname{p}}}$ horizontal sections described above. Furthermore, this relative moduli space carries a relative obstruction theory. Let $E$ be a stable sheaf on the fiber $X$ over $s\in S$ then the tangent space is given by $\Ext^1_X(E,E)_0$ and the obstruction space by $\Ext^2_X(E,E)_0$.

    Furthermore, the same is true for a family of K3 surfaces.
\end{proposition}

\begin{remark}
\label{rmk:existence-mod-space-enriques}
    We note that fixing a Mukai vector $v$ in a single fiber ensures that it extends to an \'etale neighborhood of that fiber. We can certainly extend $v_0$ and $v_2$. For $v_1$, we can assume that it is given by $c_1(L)$ for some $L$, a line bundle over a fiber.(otherwise there are no sheaves with Mukai vector $v$). By Lemma~\ref{lem:deform-line-bundle}, $L$ extends to a line bundle $\tilde{L}$ over some neighborhood $U\subseteq S$. By the Proposition 3.33 of \cite{bae2022countingsurfacescalabiyau4folds}, there is a horizontal section $\tilde{\Ch}_1(\tilde{L})\in\Gamma(U,F^1 \mathcal{H}_{D R}^{2}(\mathcal{Y}_U / U))$ such that over any $\C$-point $b$, $\tilde{\Ch}_1(\tilde{L})_b=\Ch(\tilde{L}_b)\in H^2(\mathcal{Y}_b,\C)$.  We can also fix the polarization in a single fiber first and the extend to a neighborhood in this way.
\end{remark}

\begin{proof}[Proof of Prop.~\ref{prp:existence-mod-space-enriques}]
    We fix $s \in S$. Then our chosen unobstructed Azumaya algebra on the fibre $\mathcal{Y}_s$ deforms and so we can shrink $S$ and assume that we are given a family of Azumaya algebras $\mathcal{A}$. Next, we note that in the K3 case, fixing a Mukai vector is the same as fixing the rank, determinant and second Chern class. So in this case, the proposition follows directly from Thm.~\ref{thm:moduli-space-existence} and Thm.~\ref{thm:perfect-obs-th}. In the Enriques case, we need to do a little extra work. In this case, the Mukai vector does not determine the determinant, since the map $\Pic(Y) \to H^2(Y, \mathbb{Q})$ identifies $L$ and $L \otimes \omega_Y$ for an Enriques surface $Y$. Pick any $L$ representing the class. Again, $L$ extends to some neighbourhood of $S$ by Lemma~\ref{lem:deform-line-bundle}. Then our desired moduli space is the union of the two moduli spaces: the one with Mukai vector $v$ and determinant $L$ and the one with determinant $L \otimes \omega_Y$ instead.
\end{proof}

\begin{definition}
    We denote by $e^{\vir}(M^H_{\mathcal{A}/Y}(v_\mathcal{A}))$ the virtual Euler characteristic of the moduli space of twisted sheaves. Note that the existence of the relative obstruction theory implies that it is deformation invariant.
\end{definition}

In the following, we will only consider a twisted Mukai vectors $v_\mathcal{A}$ of rank $2$. By Remark \ref{rk2onlystable}, in rank $2$ case, all twisted sheaves are stable, therefore we can drop the gcd condition. 
The goal of this paper is to compute the virtual Euler characteristic. By deformation invariance and Theorem~\ref{thm:deform-to-beauville-locus}, we can do so with the additional assumption that $Y \in B_N$ with $N < -\frac{v_\mathcal{A}^2+5}{2}$. The reason for this assumption is that we can use the following theorem of Reede~\cite{reede2024enriquessurfacestrivialbrauer}. We note that Reede considers ``generically simple torsion-free modules'', but for $\mathcal{A}$ of degree 2 with nontrivial Brauer class $B_Y$, this condition is the same as $r = 2$ and torsion free.

\begin{theorem}
\label{thm:reede}
Let $v_\mathcal{A}$ be a twisted Mukai vector with rank $2$.
    Assume that $N < -\frac{v_\mathcal{A}^2+5}{2}$. Then $M_{\mathcal{A}/Y}^H(v_\mathcal{A})$ is smooth of the expected dimension. Furthermore, the K3 cover $\pi : \tilde{Y} \to Y$ induces a morphism of moduli spaces of twisted sheaves
    \begin{equation}
    \label{eq:pullback-twisted}
    M_{\mathcal{A}/Y}^H(v_\mathcal{A}) \to M_{\pi^*\mathcal{A}/\tilde{Y}}^{\pi^*H}(\pi^*v_\mathcal{A})
    \end{equation}
    which is an \'etale two-to-one cover onto its image. The image is identified with $\mathrm{Fix}(\iota^*)$, where $\iota^*$ acts by pullback on $M_{\tilde{Y}}^{\pi^*H}(\pi^*v_\mathcal{A})$. 
\end{theorem}

Since we are on the Beauville locus, $\pi^*\mathcal{A}$ is a trivial algebra that can be identified with $\End(\mathcal{O}_{\tilde{Y}} \oplus L)\cong\End(\mathcal{O}_{\tilde{Y}} \oplus L^{-1})$ with $L$ a line bundle satisfying $\iota^*L \cong L^{-1}$ and $c_1(L)^2 < 4N+2$. In this setting, by the Proposition \ref{MoritaEq}(2), we have $F \mapsto (\mathcal{O}_{\tilde{Y}} \oplus L^{-1}) \otimes F$ is an equivalence from the category of untwisted sheaves to the category of twisted sheaves. Since we assume the gcd condition, this preserves stability and so defines an isomorphism
\begin{equation}
    \label{eq:moduli-trivialisation}
\Theta : M^{\pi^*H}_{\pi^*\mathcal{A}/\tilde{Y}}(\pi^*v_\mathcal{A}) \cong M^{\pi^*H}_{\tilde{Y}}(\pi^*v_\mathcal{A} \otimes e^{c_1(L)/2}).
\end{equation}

We introduce the following operation, following Reede.

\begin{definition}
\label{sigma}
    For any sheaf $F$ on $\tilde{Y}$, we define $\sigma(F) = \iota^*F \otimes L$.
\end{definition}

Reede shows that $\Theta$ intertwines the action of $\iota$ on the left hand side of~\eqref{eq:moduli-trivialisation} and $\sigma$ on the right hand side. Therefore, we can conclude the following.

\begin{proposition}
\label{prp:reede-untwisted}
Let $v_\mathcal{A}$ be a twisted Mukai vector with rank $2$.
Let $Y$ be an Enriques surface with polarisation $H$. Assume that $Y \in B_N$ with $N < -\frac{v_\mathcal{A}^2+5}{2}$. Then there is a morphism
    \[
    M_{\mathcal{A}/Y}^H(v_\mathcal{A}) \to M_{\tilde{Y}}^{\pi^*H}(\pi^*v_\mathcal{A} \otimes e^{c_1(L)/2})
    \]
    which is \'etale two-to-one to its image. The image is given by the fixed locus of $\sigma$. Also, the fixed locus $\operatorname{Fix}(\sigma)\subseteq M_{\tilde{Y}}^{\pi^*H}(\pi^*v_\mathcal{A} \otimes e^{c_1(L)/2})$ is a 
    smooth Lagrangian subscheme if it is not empty
\end{proposition}

\begin{remark}
\label{rem:cover-odd-even}
In \cite{Kim_1998}, Kim considered the moduli of $\mu$-stable untwisted sheaves (of any rank) on Enriques surface and shown that the image of the pullback map to K3 is a Lagrangian subvariety of the moduli of stable sheaves on K3. The image of this pullback map is the fixed locus of the involution $\iota^\ast$ on the moduli of K3.
The pullback map is a two-to-one from the smooth locus of moduli on Enriques.

Assume now that moduli on Enriques is smooth and we have an unbranched two-to-one covering from the moduli on Enriques to $\operatorname{Fix}(\iota^\ast)$.
    We note that the behavior of the two-to-one cover depends on the parity of the rank of the sheaves under consideration. Note that for any stable sheaf $E$, the sheaf $E \otimes \omega_Y$ is still $\mu$-stable. This has an interesting interaction with the fact that the moduli space in the Enriques case is a disjoint union of moduli spaces with fixed determinant.

    In the odd case, We have $\det(E\otimes \omega_Y)=\det(E)\otimes \omega_Y$, the determinant is changed by a factor $\omega_Y$. Hence, the two-to-one cover induced by the pullback is a trivial one: we have
    \[
    M^H_Y(2k+1, L, n) \sqcup M^H_Y(2k+1, L \otimes \omega_Y, n) \to \mathrm{Fix}(\iota^*)
    \]
    where both components are isomorphic to the image. However, in the even case, the determinant stays the same. The two preimages of a point in $\operatorname{Fix}(\iota^\ast)$ by the covering map are both in $M_Y^H(2k, L, n)$ or are both in $M_Y^H(2k, L\otimes \omega_Y, n)$.  This means that the images of $M_Y^H(2k, L, n)$ and $M^H_Y(2k, L \otimes \omega_Y, n)$ in $\mathrm{Fix}(\iota^*)$ are actually disjoint and so $\mathrm{Fix}(\iota^*)$ splits up as the union of two disjoint components, both of which admit a two-to-one cover from the Enriques moduli space.

    This situation mirrors precisely the case for rank $2$ twisted sheaves. One may also hope this mirrors for moduli of higher rank twisted sheaves. The structure of the covering maps respectively for $M^H_Y(2, L \otimes \omega_Y, n)$ and $M^H_Y(2, L, n)$ is an interesting problem.
\end{remark}

\begin{remark}
\label{odddimrmk}
Within the setting of Proposition \ref{prp:reede-untwisted}, let us denote the Mukai vector $\pi^*v_\mathcal{A} \otimes e^{c_1(L)/2}$ by $\overline{v}$. $\overline{v}$ has the form $\overline{v}=(2,\overline{D}+L,n)$ where $\overline{D}$ means the pullback of some $D\in \operatorname{NS}(Y)$ to $\widetilde{Y}$ by $\pi$ and $n$ is an integer. Then the dimension of $M_{\widetilde{Y}}^{\pi^\ast H}(\overline{v})$ is given by 
\begin{align}
   \dim M_{\widetilde{Y}}^{\pi^\ast H}(\overline{v}) -2&= (\overline{v},\overline{v})\\
    &=-\int_X(2,\overline{D}+L,n)\cdot (2,-\overline{D}-L,n)\\
    &=-4n+(\overline{D}+L)^2\\
    &=-4n+\overline{D}^2+L^2 \equiv 2 \mod 4 
\end{align} 
Therefore $\dim M_{\widetilde{Y}}^{\pi^\ast H}(\overline{v})$ is always divisible by $4$, thus $\dim \operatorname{Fix}(\sigma)$ is an even number since it is a Lagrangian. This implies $\dim M^H_{\mathcal{A}/Y}(v_\mathcal{A})$ is an even number. This means with assumptions in Proposition \ref{prp:reede-untwisted}, the moduli spaces $\dim M^H_{\mathcal{A}/Y}(v_\mathcal{A})$ with odd dimensions are empty and have topological Euler characteristic $0$.
\end{remark}

The goal of this paper is to compute $e^{\vir}(M^H_{\mathcal{A}/Y}(v_\mathcal{A}))$. To do this, we will first compute it in the setting of the Proposition \ref{prp:reede-untwisted} and then use the deformation invariance of the virtual Euler characteristic to deform to other Enriques surfaces. In the setting of the Proposition \ref{prp:reede-untwisted}, $e^{\vir}(M^H_{\mathcal{A}/Y}(v_\mathcal{A}))=e(M^H_{\mathcal{A}/Y}(v_\mathcal{A})) = 2\cdot e(\operatorname{Fix}(\sigma))$ since $M^H_{\mathcal{A}/Y}(v_\mathcal{A})$ is smooth and it is the two-to-one cover of $\operatorname{Fix}(\sigma)$. To compute $e(\operatorname{Fix}(\sigma))$, the ideal will be using the Lefschetz formula in the following proposition. 

\begin{proposition}[Lefschetz formula]
\label{lefschetz}
Let $X$ be a smooth proper variety over $\C$ and $f$ an endomorphism. Assume the fixed locus of $f$ is smooth and proper, let $X^f$ denote the fixed locus then:
\begin{equation}
e(X^f)=\sum_{i=0}^{2\operatorname{dim} X}(-1)^i \operatorname{Tr}\left(f^*: H^i( X,\C) \rightarrow H^i( X,\C) \right),
\end{equation}
    where $f^\ast:H^\ast(X,\C)\to H^\ast(X,\C)$ is the pullback.
\end{proposition}
\begin{proof}
By Chapter 16 of \cite{fulton_1998_intersection},
the graph of $f$ gives a correspondence $F$ from $X$ to $X$ which is a cycle on $X\times X$. Let $\Delta$ be the cycle of the diagonal. Then, by Example 16.1.15 in \cite{fulton_1998_intersection}, the intersection class $F\cdot \Delta$ is given by:
\begin{equation}
\int_{X \times X} F \cdot \Delta=\sum_i(-1)^i \operatorname{Tr}\left(f^*: H^i( X,\C) \rightarrow H^i( X,\C) \right) .
\end{equation}

Proposition 16.2 in \cite{fulton_1998_intersection} shows:
\begin{equation}
\int_{X \times X} F \cdot \Delta= \int_{\Delta=X}  c(T_\Delta) \cdot c(N_{F \cap \Delta/\Delta})^{-1} 
= e(F \cap \Delta)=e(X^f),
\end{equation}
where $N_{F \cap \Delta/\Delta}$ denotes the normal bundle of $F \cap \Delta$ in $\Delta$.
\end{proof}

The Lefschetz formula will reduce the problem to calculating the trace of $\sigma^\ast:H^\ast(M,\C)\to H^\ast(M,\C)$, where $M$ denote the space where $\sigma$ acts. However this trace $\operatorname{tr}(\sigma^\ast)$ is still difficult to compute. To calculate it we want to find some ring isomorphism $\phi:H^\ast(M,\C)\to H^\ast(\widetilde{Y}^{[N]},\C)$, where $\widetilde{Y}^{[N]}$ denotes the Hilbert scheme of $N$ points on the covering K3 surface $\widetilde{Y}$ with the same dimension as $M$. 
We also want to find the following commutative diagram:
    \[\begin{tikzcd}
H^\ast(M,\C) \arrow{r}{\sigma^\ast} \arrow[swap]{d}{\phi} & H^\ast(M,\C) \arrow{d}{\phi} \\%
H^\ast(X^{[N]},\C) \arrow{r}{\iota}& H^\ast(X^{[N]},\C),
\end{tikzcd}\]   
such that the fixed locus of the bottom action $\iota$ is well understood. In fact, this ring isomorphism $\phi$ are constructed via the Markman operators, which will be introduced in the next section.

\section{Markman operators} \label{sec:markman-operators}
In \cite{markman2005monodromy}, Markman defined operators relating cohomology rings of two K3 surfaces. In this section, we will recall some facts about Markman operators, especially the universality result of \cite{Oberdieck_2022} and some applications in \cite{wang2024virasoroconstraintsk3surfaces}. This section will not involve Enriques surfaces and so we let $X$ be an arbitrary K3 surface.

Let $X$ be a non-singular and projective K3 surface. There is a bilinear form, called the $\textit{Mukai pairing}$, on $\Lambda:= H^\ast(X,\mathbb{Z})$ defined as:
\begin{align*}
    (x,y):=-\int_X x^\vee y, 
\end{align*}
where $\bullet^\vee$ is defined as follows: if one writes $x = (r,D,n)$ as the decomposition of degree then $x^\vee = (r,-D,n)$. This pairing is symmetric, unimodular, of signature $(4,20)$ and the resulting lattice is called the $\textit{Mukai lattice}$.

For a coherent sheaf $\mathcal{F}$ on $X$, its Mukai vector is defined to be $v(\mathcal{F})=\operatorname{ch}(\mathcal{F})\cdot \sqrt{\operatorname{td}_X}$.
Let $M:=M^H_{X}(v)$ be the moduli space of $H$-semistable sheaves on $X$ with Mukai vector $v\in\Lambda$. For a $v\in\Lambda$, I will write $v=(v_0,v_1,v_2)$ to denote components in even degrees of $H^\ast(X,\mathbb{Z})$.  Following \cite{markman2005monodromy}, we define:
\begin{definition}
A non zero class $v\in\Lambda$ is called \textit{effective}, if $(v,v) \geq -2$,
$v_0 \geq 0$, and the following conditions hold. If $v_0 = 0$, $v_1$ needs to be the class of an effective (or trivial) divisor on $X$. If both $v_0$ and $v_1$ vanish, $v_2$ needs to satisfy $v_2 >0$. The class $v$ is \textit{primitive}, if it is not a multiple of a class
in $\Lambda$ by an integer larger than 1.
\end{definition}
For a primitive and effective $v$ there is a polarisation $H$, called $v$-suitable in \cite{markman2005monodromy}, such that $M$ does not contain any strictly semistable sheaves and is smooth projective and also admits a (twisted) universal sheaf. In this case we have: $\dim M=(v,v)+2$.

Define the morphism $B:H^\ast(X,\mathbb{Q})\to H^\ast(M,\mathbb{Q})$ as
\begin{align}
\label{Bmorphism}
    B(x)=\pi_{M\ast}(u_v\cdot x^\vee),
\end{align}
 with \[u_v:=\exp\left( \frac{\theta_{\mathcal{F}}(v)}{(v,v)}\right)\cdot 
\operatorname{ch}(\mathcal{F})\cdot\sqrt{\text{td}_X} \in H^\ast(M\times X,\mathbb{Q}), \]
and 
\[\theta_\mathcal{F}(x)=[\pi_{M\ast}(\operatorname{ch}(\mathcal{F})\pi_X^\ast(\sqrt{\text{td}_X}\cdot x^\vee))]_2,\]
where $[ \bullet ]_k$ means take the real degree $k$ component of a cohomology class.
In the definition, some pullbacks $\pi_M^*$, $\pi_X^*$ have been suppressed. One can check that $B$ is independent of the choice of the (twisted) universal sheaf $\mathcal{F}$. We also write $B_k(x)$ for the component in cohomological degree $2k$.

Let $X_1$ and $X_2$ be two non-singular projective K3 surfaces with polarizations $H_1,H_2$. Let $g:H^\ast(X_1,\mathbb{Z})\to H^\ast(X_2,\mathbb{Z})$ be an isometry of Mukai lattices, assume $v_1$, $v_2$ are two vectors in the Mukai lattice satisfying the assumptions of the previous paragraph and $v_2=g(v_1)$. Let $M_i:=M_{H_i}(v_i)$ for $i=1,2$. 
Markman defined the transformation $\gamma(g): H^\ast(M_1,\mathbb{C}) \to H^\ast(M_2,\mathbb{C})$ in \cite{markman2005monodromy}.
The main properties of $\gamma(g)$ are given in the following theorem: 

\begin{theorem}[Markman]
\label{markmantheorem}
Let $X_1,X_2$ and $v_1,v_2$ as above. For any isometry $g:H^\ast(X_1,\mathbb{C})\to H^\ast(X_2,\mathbb{C})$ such that $g(v_1)=v_2$, $\gamma(g)$ is the unique operator such that: 
\begin{enumerate}
\item[$\mathrm{(i)}$] $\gamma(g)$ is a degree-preserving ring isomorphism and is an isometry with respect to the Poincaré pairing: $\langle x,y\rangle = \int_M xy$ for all $x,y \in H^\ast(M,\mathbb{Q})$. 
\item[$\mathrm{(ii)}$] $\gamma(g_1)\circ \gamma(g_2) = \gamma(g_1g_2)$ and $\gamma(g)^{-1} = \gamma(g^{-1})$(if it makes sense), where $g_1,g_2$ are two isometries.
\item[$\mathrm{(iii)}$] $\gamma(g)\left(c_k(T_{M_1})\right)=c_k(T_{M_2})$.
\item[$\mathrm{(iv)}$]
$\gamma(g)(B(c))=B(g(c)))$ for all $c\in H^\ast (X_1,\mathbb{Q})$ 
\end{enumerate}
\end{theorem}

Using Markman's operator, Oberdieck proved a universality property for integrals over $M$; this roughly means that the integral $\int_M P(B_i(x_j),c_l(T_M))$ only depends on the polynomial $P$, the dimension of $M$ and pairings $(v,x_i)$, $(x_i,x_j)$ for all $i,j$ i.e.~the intersection matrix
\begin{align*}
    \begin{pmatrix}
(v,v) & (v,x_i)_{i=1}^k \\
(x_i,v)_{i=1}^k & (x_i,x_j)_{i,j=1}^k
\end{pmatrix}.
\end{align*}

\begin{theorem}[Oberdieck \cite{Oberdieck_2022}]
\label{theoremUnivers}
let $P(t_{ij},u_r)$ be a polynomial depending on the variables
$t_{ij}$, $j = 1,...,k,$   $i \geq 0$, and $u_l$, $l\geq 1$. Let also $A = (a_{ij})^k_{i,j=0}$ be a $(k + 1) \times (k + 1)$-matrix. Then there exists $I(P,A)\in \mathbb{Q}$ ($I(P,A)$ is a rational number only depending on $P,A$) such that for any $M=M_H(v)$ with $dim(M)>2$ and for any $x_1,...,x_k\in \Lambda$ with 
\[
A=\begin{pmatrix}
(v,v) & (v,x_i)_{i=1}^{k} \\
(x_i,v)_{i=1}^k & (x_i,x_j)_{i,j=1}^k
\end{pmatrix},\]
we have
\[\int_M P(B_i(x_j),c_l(T_M)) = I(P,A).\]
\end{theorem}

Before sketching the proof, let us first recall the following lemma in \cite{Oberdieck_2022}, which will be used later.
\begin{lemma}[Lemma 2.13 of \cite{Oberdieck_2022} ]
\label{Oberdlemma}
    Let $V$ be a finite-dimensional $\mathbb{C}$-vector space with a $\mathbb{C}$-linear inner product $\langle\bullet,\bullet\rangle$(latter when use this lemma, we will take the inner product induced by the Mukai pairing). Let $v_1,\dots,v_k\in V$ and $w_1,\dots,w_k\in V$ be lists of vectors such that
    \begin{enumerate}
    \item[$\mathrm{(i)}$] $\operatorname{Span}(v_1,\dots,v_k)$ is non-degenerate with respect to $\langle\bullet,\bullet\rangle$,
    \item[$\mathrm{(ii)}$] $\operatorname{Span}(w_1,\dots,w_k)$ is non-degenerate with respect to $\langle\bullet,\bullet\rangle$, 
    \item[$\mathrm{(iii)}$] $\langle v_i,v_j \rangle = \langle w_i,w_j \rangle$ for all $i,j$. 
    \end{enumerate}
    Then there exists an isometry $\phi:V\to V$ such that $\phi(v_i)=w_i$ for all $i$.
 \end{lemma}

The idea of the proof of Theorem \ref{theoremUnivers} is the following: Given the data $(M_X^H(v),x_i)$ as above, Oberdieck \cite{Oberdieck_2022} shows that, assuming $\operatorname{dim}(M)>2$, there exists $y_i\in \Lambda_{\mathbb{C}}$ which have the same intersection matrix as above, and satisfy
\[\int_{M_X^H(v)}P(B_i(x_j),c_l(T_M))=\int_{M_X^H(v)}P(B_i(y_j),c_l(T_M))\]
and $\operatorname{Span}(v,y_1,\dots,y_k)$ is a non-degenerate subspace of $\Lambda_{\mathbb{C}}$ (i.e. the restriction of the inner product of $\Lambda_{\mathbb{C}}$, induced by the Mukai pairing, onto the subspace is non-degenerate).

Therefore, given two arrays of vectors $\{v,x_1,\dots,x_k \}$ and $\{v',x_1',\dots,x_k'\}$ having the same intersection matrix, we can always assume that $\operatorname{Span}(v,x_1,\dots,x_k)$ and $\operatorname{Span}(v',x_1',\newline\dots,x_k')$ are non-degenerate. By Lemma \ref{Oberdlemma}, there exists an isometry $g:H^\ast(X,\mathbb{C})\to H^\ast(X',\mathbb{C})$ taking $(v,x_1,\dots,x_k)$ to $(v'=g(v),x_1'=g(x_1),\dots,x_k'=g(x_k))$. 
 Therefore by the properties of the Markman operator and morphism $B$ one has:
 \begin{align*}
     \int_{M_X^H(v)} P(B_i(x_j),c_l(T_{M(v)})) &=  \int_{M_{X'}^{H'}(v)} \gamma(g)P(B_i(x_j),c_l(T_{M(v')}))\\
     &=  \int_{M_{X'}^{H'}(v)} P(B_i(gx_j),c_l(T_{M(v')}))\\
     &=  \int_{M_{X'}^{H'}(v)} P(B_i(x_j'),c_l(T_{M(v')})).
 \end{align*}

In the next section, we will need an application of the Theorem~\ref{theoremUnivers} 
in \cite{wang2024virasoroconstraintsk3surfaces}. We start by introducing the notion of descendent algebra. The descendent algebra $\mathbb{D}^X$ is the commutative algebra generated by symbols of the form:
\[\ch{i}{\gamma}\quad\text{ for }i\geq 0, \gamma\in H^\ast(X,\C)\] subject to the linearity relations\[\ch{i}{\lambda_1\gamma_1+\lambda_2\gamma_2}=\lambda_1\ch{i}{\gamma_1}+\lambda_2\ch{i}{\gamma_2}\] for $\lambda_1,\lambda_2\in \mathbb{C}$. 

Let $v$ be a Mukai vector satisfying the assumptions at the beginning of the Section \ref{sec:markman-operators}. Consider the moduli space $M:=M_H(v)$ with $r:=\operatorname{rk}(v)\geq 1$ and the product $M\times X$, let $\pi_M$ and $\pi_X$ be the projection from $M\times X$ to $M$ and $X$. Let $\mathcal{F}$ be a (twisted) coherent sheaf on $M\times X$. The geometric realization with respect to $\mathcal{F}$ on $M\times X$ is the algebra homomorphism 
\begin{align}
\label{geomreal}
    \xi_{\mathcal{F}}:\mathbb{D}^X \to H^\ast(M,\C),
\end{align}
which acts on generators $\ch{i}{\gamma}$ with $\gamma\in H^{\ast}(X,\C)$ as 
\[\xi_{\mathcal{F}}\left( \ch{i}{\gamma} \right)=[\pi_{M\ast}\left( \operatorname{ch}(\mathcal{F}) \pi_X^\ast \gamma \right)]_{2i},\]
where $[\bullet]_k$, as before, denotes the cohomological degree $k$ part.
We now rephrase Lemma 4.1 and Lemma 4.2 of \cite{wang2024virasoroconstraintsk3surfaces} into the following form: 

\begin{lemma}[Lemma 4.1 and Lemma 4.2 in \cite{wang2024virasoroconstraintsk3surfaces}]
\label{claim}
Consider a Mukai vector $v=(r,D,n)$ with $r\geq 1$ and $(v,v)\geq 0$. Let $\{L_i\}_{1\leq i\leq 22}\subseteq H^2(X,\C)  \text{ be a basis of } H^2(X,\C)$. Denote the fundamental class and the point class of $X$ respectively by $\mathbbm{1}$ and $\boldsymbol{\mathrm{p}}$. Define an algebra isomorphism $\phi:\mathbb{D}^X\mapsto \mathbb{D}^X$ by \begin{align}
    \phi(\ch{i}{L_j})&=\ch{i}{L_j},\forall L_j \in \{L_k\}_k, \forall i\\
\phi(\ch{i}{\boldsymbol{\mathrm{p}}})&=r\ch{i}{\boldsymbol{\mathrm{p}}}, \forall i\\
\phi(\ch{i}{\mathbbm{1}+\boldsymbol{\mathrm{p}}})&=\frac{1}{r}\ch{i}{\mathbbm{1}+\boldsymbol{\mathrm{p}}} , \forall i.
\end{align}
Let $\mathcal{F}$ be a universal sheaf on $M(v)\times X$. Let $N$ be an integer such that $2N-2=(v,v)=D^2-2rn$. The universal sheaf on $X^{[N]}\times X$ is $\mathcal{I}_\mathcal{Z}=\mathcal{O}_{X^{[N]}\times X}-\mathcal{O}_{{\mathcal{Z}}}$

    Then we have:
$\phi$ descends to cohomology, i.e., there is a ring isomorphism $ \tilde{\phi}$ such that the following diagram commute:
\[
    \begin{tikzcd}
\mathbb{D}^X \arrow{r}{\phi} \arrow[swap]{d}{\xi_{\mathcal{F}\otimes \det{(\mathcal{F})}^{-1/r}}} & \mathbb{D}^X \arrow{d}{\xi_{\mathcal{I}_\mathcal{Z}}} \\%
H^\ast(M(v),\C) \arrow{r}{\Tilde{\phi}}& H^\ast(S^{[N]},\C).
\end{tikzcd}\]

We also have the following equality:
\begin{align}  
\int_{M(v)} \xi_{\mathcal{F}\otimes \det{(\mathcal{F})}^{-1/r}}(D)
=\int_{X^{[N]}} \xi_{\mathcal{I}_\mathcal{Z}}( \phi(D)),\forall D\in\mathbb{D}^X.
\end{align}

\end{lemma}

\begin{proof}
We sketch the proof in \cite{Oberdieck_2022} and in \cite{wang2024virasoroconstraintsk3surfaces} here. 

In the $\operatorname{dim}(M(v))>2$ case:
we first write $\xi_{\mathcal{F}\otimes \det{(\mathcal{F})}^{-1/r}}(\Ch(\lambda))$ in the form of a product of classes of the form $B(\bullet)$, where $\lambda\in H^\ast(X,\C)$:
\begin{align}
\label{Bform}
    \xi_{\mathcal{F}\otimes \det{(\mathcal{F})}^{-1/r}}(\Ch(\lambda))=\exp{\left( B_1\left(\frac{\boldsymbol{\operatorname{p}}}{(\boldsymbol{\operatorname{p}},v)}-\frac{v}{(v,v)}\right) \right) B\left(\left(\exp{\left(\frac{c_1(v)}{r} \right)} \lambda^\vee \sqrt{\text{td}_X}^{-1}\right)\right)}.
\end{align}
 We need to look at the intersection matrix of the classes appearing in the arguments of $B$ and $B_1$ in \eqref{Bform}. In fact this intersection matrix is equivalent to the intersection matrix of the following elements.
\begin{align}
\label{classes}
    L_i \exp{\left(\frac{c_1(v)}{r} \right)},\quad\text{with }L_i\in \{L_k\}_k;\quad \exp{\left(\frac{c_1(v)}{r} \right)};\quad \boldsymbol{\operatorname{p}};\quad\frac{\boldsymbol{\operatorname{p}}}{r};\quad v.
\end{align}
In order to find $\tilde{\phi}$ relating $M(v)$ and $X^{[N]}$, we want to find an an isometry $h:H^\ast(X,\C)\to H^\ast(X,\C)$ sending the above classes in \eqref{classes} to the following classes:
\begin{align}
\label{classes2}
    L_i,\quad\text{with }L_i\in \{L_k\}_k;\quad \frac{\mathbbm{1}}{r};\quad r\boldsymbol{\operatorname{p}};\quad\boldsymbol{\operatorname{p}};\quad \mathbbm{1}-(N-1)\boldsymbol{\operatorname{p}}.
\end{align}
As the calculation in \cite{wang2024virasoroconstraintsk3surfaces} shows, \eqref{classes} and \eqref{classes2} have the same intersection matrices and their spans are both non-degenerate, therefore the Lemma \ref{Oberdlemma} implies such an isometry $h$ exists and therefore we could take $\tilde{\phi}$ as $\gamma(h)$, the Markman operator induced by $h$. And one can work out that the action of $\tilde{\phi}$ on the descendent algebra is $\phi$.

If $\operatorname{dim}(M(v))=2$ then $M$ and $X^{[1]}\cong X$ are K3 surfaces, one can take $\tilde{\phi}$ as the isometry in Section 2.4 of \cite{Oberdieck_2022}. 
\end{proof}

\section{The topological Euler characteristic of the fixed locus} \label{sec:top-eul-char}
Let $v=(v_0,v_1,v_2)$ be a Mukai vector with $v_1$ of type $(1,1)$. To simplify the notation, instead of using $\widetilde{Y}$, we will use $X$ to denote the covering K3 surface of the Enriques $Y$. Assume that $v$ is effective and primitive and assume the polarization is $v$-suitable as explained in the beginning of Section \ref{sec:markman-operators}. Let $\mathcal{F}_v$ be a universal sheaf on $M_X^H(v)\times X$.
We define $\mathcal{G}_v:=\mathcal{F}_v\otimes\det(\mathcal{F}_v)^{-1/r}\otimes \pi_X^\ast\Delta^{1/r}$, where $\Delta$ is a line bundle on $X$ with $c_1(\Delta)=v_1$. We also define $\mathcal{K}_v:=\mathcal{F}_v\otimes\det(\mathcal{F}_v)^{-1/r}$. Note that $\mathcal{G}_v$ and $\mathcal{K}_v$ exist as rational K-theory classes and $\mathcal{G}_v$ exists as a twisted sheave\footnote{We need $\mathcal{G}_v$ to be a twisted universal sheaf, but $\mathcal{K}_v$ is sufficient to be a K-theory class.}. We will omit the subscript $v$ when there is no ambiguity. We will use the following notation
$\mathrm{ch}^{\mathcal{G}}_i(\gamma):=\xi_\mathcal{G}(\mathrm{ch}_i(\gamma))$ and similarly for $\mathcal{K}$. Notice that $\mathcal{G}$ is a universal sheaf but $\mathcal{K}$ is not.

Now, we want to prove the following lemma:
\begin{lemma}
\label{deslem}
Let $X$ be a K3 surface.
    Consider the following commutative diagram:
    \[
    \begin{tikzcd}
\mathbb{D}^X \arrow{r}{f} \arrow[swap]{d}{\phi} & \mathbb{D}^X \arrow{d}{\phi} \\%
\mathbb{D}^X \arrow{r}{g}& \mathbb{D}^X,
\end{tikzcd}\]
where all maps are descendent algebra isomorphisms. Consider the moduli space $M(v):=M_X^H(v)$ with $v=(v_0,v_1,v_2)$ and the (twisted) sheaf $\mathcal{K}_v$ defined above. We will use the geometric realization map $\xi_v:=\xi_{\mathcal{K}_v}:\mathbb{D}^X\to H^\ast(M(v),\C)$.

Assume $f$ descends to cohomology with respect to $\xi_v$; i.e., there exists a ring isomorphism $\Tilde{f}$ making the following diagram commute:
\[
    \begin{tikzcd}
\mathbb{D}^X \arrow{r}{f} \arrow[swap]{d}{\xi_v} & \mathbb{D}^X \arrow{d}{\xi_v} \\%
H^\ast(M(v),\C) \arrow{r}{\Tilde{f}}& H^\ast(M(v),\C).
\end{tikzcd}\]
Assume also that there exist $M(w):=M_X^{H'}(w)$ such that $\phi$ satisfies the following property:
\begin{align}
\label{desCond}
    \int_{M(v)}\xi_v(D) = \int_{M(w)}\xi_w(\phi(D)),\quad \forall D\in \mathbb{D}^X,
\end{align}
where $\xi_w$ is defined similarly as above.
Then $g$ also descends to cohomology with respect to $\xi_w$.
\end{lemma}
\begin{proof}
It is known that for K3 surfaces the map $\xi_\mathcal{F}:\mathbb{D}^X\to H^\ast(M(v),\C)$ is surjective \cite{MARKMAN2007622}, for any universal sheaf $\mathcal{F}$. $\xi_v$ uses $\mathcal{F}\otimes\det(\mathcal{F})^{-1/r}$ which is not a universal sheaf. However it is easy to see that $\mathcal{G}=\mathcal{F}\otimes\det(\mathcal{F})^{-1/r}\otimes \Delta^{1/r}$ is a universal sheaf, where $\Delta$ is a line bundle on the surface with $c_1(\Delta)=v_1$. Therefore we have that $\xi_{\mathcal{G}}(\ch{i}{\gamma})= \xi_v(\ch{i}{e^{\frac{1}{r}v_1}\cdot \gamma})$ which implies the map $\xi_v:\mathbb{D}^X\to H^\ast(M(v),\C)$ is surjective. Denote the kernel of $\xi_v$ by $I_v$. We then have a ring isomorphism $\mathbb{D}^X/I_v\cong H^\ast(M(v),\C)$. The statement that $f$ descends to $\Tilde{f}$ is equivalent to saying that the kernel of $\mathbb{D}^X\xrightarrow{f} \mathbb{D}^X\to \mathbb{D}^X/I_v$ contains $I_v$. Therefore we only need to prove $\ker(\mathbb{D}^X\xrightarrow{g} \mathbb{D}^X\to \mathbb{D}^X/I_w)\supseteq I_w$.

The ideal $I_v$ can be written as follows:
\begin{align}
    I_v=\left\{D\in \mathbb{D}^X\bigg|\int_{M(v)}\xi_v(DD')=0\quad\forall D'\in \mathbb{D}^X\right\}.
\end{align}
This can be seen as follows: Clearly, if $D$ is in the kernel, it is in $I_v$. In the other direction, if $D \in I_v$, then the image of $D$ in $H^\bullet(M(v))$ is zero due to the perfect pairing $H^\bullet(M(v)) \times H^\bullet(M(v)) \to \mathbb{Q}$.

By the property of $\phi$, it is obvious that $\phi(I_v)=I_w$. Therefore we have:
\begin{align}
    \ker(\mathbb{D}^X\xrightarrow{g} \mathbb{D}^X\to \mathbb{D}^X/I_w) &= \left\{D\in \mathbb{D}^X|\phi f \phi^{-1}(D)\in I_w\right\}\\&=\left\{D\in \mathbb{D}^X\bigg|\int_{M(w)}\xi_w\left(D'\cdot\phi f \phi^{-1}(D)\right)=0\quad\forall D'\in \mathbb{D}^X\right\}\\
    &=\left\{D\in \mathbb{D}^X\bigg|\int_{M(v)}\xi_v\left(\phi^{-1}(D')f(\phi^{-1}(D))\right)=0\quad\forall D'\in \mathbb{D}^X\right\}\\
    &=\phi\left(\left\{D\in \mathbb{D}^X\bigg|\int_{M(v)}\xi_v\left(D'f(D)\right)=0\quad\forall D'\in \mathbb{D}^X\right\} \right)\\
&=\phi\left(\ker(\mathbb{D}^X\xrightarrow{f} \mathbb{D}^X\to \mathbb{D}^X/I_v\right)
    \supseteq\phi(I_v)=I_w.
\end{align} 
    Therefore $g$ also descents to cohomology.
\end{proof}

Now, let us recall the setting of our problem to solve. $\pi: X\to Y$ is the two-to-one covering of the K3 surface to the Enriques surface. Here, $Y$ is an Enriques surface such that $\pi^\ast B_Y=0$, where $B_Y$ is the only non-trivial element in $\operatorname{Br}(Y)$. In this case there exists an line bundle $L$ on $X$ such that $\iota^\ast L=L^{-1}$ where $\iota$ is the involution on $X$. We want to study the moduli space of $\sigma$-invariant stable sheaves of rank $2$ on $X$, where $\sigma$ is the involution defined in the Definition \ref{sigma}.
 Let $G$ be a $\sigma$-invariant sheaf of any rank, then by the calculation in \cite{reede2024enriquessurfacestrivialbrauer}, the Mukai vector of $G$ has the following form:
\begin{align}
    v(G) = (2s,\overline{D}+sL,\chi(G)-2s),
\end{align}
For some $D\in \operatorname{NS}(Y)$ and $s\in \mathbb{N}$ and $\overline{D}$ means $\pi^\ast D$. Since we only consider rank $2$ case, we set $s=1$ and $v=(2,\overline{D}+L,n)$. We denote the moduli space $M_X^H(v)$ by $M(v)$. We also require that $Y$ lies in the Beauville locus $B_N$ with $N$ in the range of the Proposition \ref{prp:reede-untwisted}. Reede \cite{reede2024enriquessurfacestrivialbrauer} shows that $\sigma$ induces a anti-sympletic involution $\sigma:M(v)\to M(v)$. Moreover $\sigma$ also induces a ring isomorphism by pullback, and we denote this isomorphism by $\sigma^\ast:H^\ast(M(v),\C)\to H^\ast(M(v),\C)$.

\begin{lemma}
\label{calofactioniota}
Let $X$ be the K3 surface and let $\iota$ be the involution as above.  Denote the moduli space of all $\mu$-stable sheaves on $X$ by $M$. $\iota$ induces an involution $\iota: M\to M$.
Consider also the operator $\Xi:\mathrm{Coh}(X)\to \mathrm{Coh}(X): E\mapsto E\otimes L$. $\Xi$ induces a morphism $\Xi:M\to M$. Let $\mathcal{E}$ be a (twisted) universal sheaf on $M\times X$. Then we have:
\begin{align}
    \iota^\ast \left(\operatorname{ch}_i^\mathcal{E}(\gamma)\right)=\operatorname{ch}_i^\mathcal{E}(\iota^\ast \gamma),\\
    \Xi^\ast\left(\operatorname{ch}_i^\mathcal{E}(\gamma)\right)=\operatorname{ch}_i^\mathcal{E}(\gamma\cdot e^{L}),
\end{align}
where $\operatorname{ch}_i^\mathcal{E}(\gamma)$ denote the geometric realization $\xi_{\mathcal{E}}\left(\operatorname{ch}_i(\gamma)\right)$.
\end{lemma}

\begin{proof}
Consider the projections $\pi_M:M\times X\to M$ and $\pi_X:M\times X\to X$. The universal sheaf $\mathcal{E}$ has the universal property that $(\iota\times \mathbbm{1}_X)^\ast \mathcal{E}\cong(\mathbbm{1}_M\times\iota )^\ast \mathcal{E}$. Then we have:
\begin{align}
    \iota^\ast(\mathrm{ch}^\mathcal{E}(\gamma))&= \iota^\ast \pi_{M\ast}(\operatorname{ch}(\mathcal{E})\cdot\pi_X^\ast \gamma)\\
    &= \pi_{M\ast} (\iota\times \mathbbm{1}_X)^\ast (\operatorname{ch}(\mathcal{E})\cdot\pi_X^\ast \gamma)\\
    &=\pi_{M\ast}  (\operatorname{ch}((\iota\times \mathbbm{1}_X)^\ast \mathcal{E})\cdot\pi_X^\ast \gamma)\\
    &=\pi_{M\ast}  (\operatorname{ch}((\mathbbm{1}_M\times\iota )^\ast \mathcal{E})\cdot\pi_X^\ast \gamma)\\
    &=\pi_{M\ast}(\mathbbm{1}_M\times\iota )^\ast(\operatorname{ch}( \mathcal{E})\cdot\pi_X^\ast \iota^\ast \gamma)\\
    &= \pi_{M\ast} (\Ch(\mathcal{E}) \cdot \pi^*_X \iota^*\gamma )\\
    &=\mathrm{ch}^\mathcal{E}(\iota^\ast\gamma).
\end{align}
For $\Xi$ we have:
\begin{align}
    \Xi^\ast(\mathrm{ch}^\mathcal{E}(\gamma))&= \Xi^\ast \pi_{M\ast}(\operatorname{ch}(\mathcal{E})\cdot\pi_X^\ast \gamma)\\
    &= \pi_{M\ast} (\Xi\times \mathbbm{1}_X)^\ast (\operatorname{ch}(\mathcal{E})\cdot\pi_X^\ast \gamma)\\
    &=\pi_{M\ast}  (\operatorname{ch}((\Xi\times \mathbbm{1}_X)^\ast \mathcal{E})\cdot\pi_X^\ast \gamma)\\
    &=\pi_{M\ast}  (\operatorname{ch}( \mathcal{E}\otimes L)\cdot\pi_X^\ast \gamma)\\
    &=\mathrm{ch}^\mathcal{E}(\gamma\cdot e^{L}),
\end{align}
where we used $(\Xi\times \mathbbm{1}_X)^\ast \mathcal{E}=\mathcal{E}\otimes L$.
\end{proof}

\begin{lemma}
\label{caldeslem}
Consider a Mukai vector $v=(2,\overline{D}+L,n)$.
    One can lift $\sigma^\ast$ to a algebra isomorphism $\Sigma^\ast$ between descendent algebras which satisfies the following commutative diagram:
\[
    \begin{tikzcd}
\mathbb{D}^X \arrow{r}{\Sigma^\ast} \arrow[swap]{d}{\xi_v} & \mathbb{D}^X \arrow{d}{\xi_v} \\%
H^\ast(M(v),\C) \arrow{r}{\sigma^\ast}& H^\ast(M(v),\C),
\end{tikzcd}\]
    where $\xi_v$ is defined as in Lemma \ref{deslem} and $\Sigma^\ast$ is defined as:
\begin{align}
\Sigma^\ast:\mathbb{D}^X\to\mathbb{D}^X:\operatorname{ch}_i(\gamma)\mapsto\operatorname{ch}_i(\iota^\ast \gamma ).
\end{align}
The above diagram can also be reformulated as:
\begin{align}
    \sigma^\ast(\mathrm{ch}^\mathcal{K}(\gamma))=\mathrm{ch}^\mathcal{K}(\iota^\ast\gamma)\quad \forall\gamma\in \mathbb{D}^X.
\end{align}
\end{lemma}
\begin{proof}
Notice that $\sigma=\Xi\circ \iota$ as in the Lemma \ref{calofactioniota}. If we consider the moduli space $M$ of $\mu$-stable sheaves on $X$ as in the Lemma \ref{calofactioniota}, then we have:
\begin{align}
    \sigma^\ast:H^\ast(M,\C)&\to H^\ast(M,\C)\\
    \mathrm{ch}^\mathcal{E}(\gamma)&\mapsto\mathrm{ch}^\mathcal{E}(\iota^\ast\gamma\cdot e^{-L}),
\end{align}
where $\mathcal{E}$ is a universal sheaf on $M\times X$.

The Mukai vectors involved in $\sigma=\Xi\circ \iota$ are $v=(2,\overline{D}+L,n)$ and $(2,\overline{D}-L,n)$, more precisely, we apply $\sigma$ to a sheaf of type $(2,\overline{D}+L,n)$, after applying $\iota$ we will get a sheaf of type $(2,\overline{D}-L,n)$, then after applying $\Xi$, we will again get a sheaf of type $(2,\overline{D}+L,n)$. We need to make sure that $M_X((2,\overline{D}-L,n))$ and $M_X((2,\overline{D}+L,n))$ are $\mu$-stable sheaves. For $M_X((2,\overline{D}+L,n))$, this is shown in the Lemma 2.5 and the Theorem 2.6 in \cite{reede2024enriquessurfacestrivialbrauer}. For $M_X((2,\overline{D}-L,n))$, we only need to notice that the discriminant of such a sheaf is the same as a sheaf of type $(2,\overline{D}+L,n)$ since $\overline{D}\cdot L=0$, then we can use the same argument as before. Therefore $M_X((2,\overline{D}-L,n))$ and $M_X((2,\overline{D}+L,n))$ are connected components of $M$. Restrict $\sigma$ to the connected component $M(v)$ and use the universal sheaf $\mathcal{G}$ defined at the beginning of this section, one has:
\begin{align}
\label{P18}
    \begin{split}
\sigma^\ast:H^\ast(M(v),\C)&\to H^\ast(M(v),\C)\\
    \mathrm{ch}^\mathcal{G}(\gamma)&\mapsto\mathrm{ch}^\mathcal{G}(\iota^\ast\gamma\cdot e^{-L}).
    \end{split}
\end{align}

Using the relation $\mathrm{ch}^\mathcal{K}(\gamma\cdot e^{\frac{1}{2}(\overline{D}+L)})=\mathrm{ch}^\mathcal{G}(\gamma)$ and $\eqref{P18}$, we find:
\begin{align}
    \sigma^\ast (\mathrm{ch}^\mathcal{K}(\gamma))&=\mathrm{ch}^\mathcal{K}(\iota^\ast[\gamma\cdot e^{-\frac{1}{2}(\overline{D}+L)}]\cdot e^{-L}\cdot e^{\frac{1}{2}(\overline{D}+L)})\\
    &=\mathrm{ch}^\mathcal{K}(\iota^\ast\gamma),
\end{align}
where the last equality used the fact that $\iota^\ast L=L^{-1}$ and $\iota^\ast \overline{D}=\overline{D}$.
\end{proof}

\begin{lemma}
\label{2diag}
Consider a Mukai vector $v=(2,\overline{D}+L,n)$. There exists ring isomorphisms $\phi$ and $\iota^\ast$ which make the following diagram commute:
\[\begin{tikzcd}
H^\ast(M(v),\C) \arrow{r}{\sigma^\ast} \arrow[swap]{d}{\phi} & H^\ast(M(v),\C) \arrow{d}{\phi} \\%
H^\ast(X^{[N]},\C) \arrow{r}{\iota^\ast}& H^\ast(X^{[N]},\C).
\end{tikzcd}\]    
Here $N$ is an even number and $\phi$ is a Markman operator thus an ring isomorphism and $\iota^\ast$ is the pull back of the cohomology ring induced by the morphism $\iota:X^{[N]} \to X^{[N]} :I\mapsto \iota^*(I)$ where $\iota$ is the involution on $X$.

\end{lemma}
\begin{proof}
    Apply Lemma \ref{claim} with Mukai vector $v=(2,\overline{D}+L,n)$, this gives an algebra isomorphism \begin{align}
\label{marphi}
\phi:\mathbb{D}^X&\to\mathbb{D}^X\\
\operatorname{ch}_i(\mathbbm{1}+\boldsymbol{\operatorname{p}})&\mapsto \frac{1}{2}\operatorname{ch}_i(\mathbbm{1}+\boldsymbol{\operatorname{p}})\\
\operatorname{ch}_i(L_j)& \mapsto \operatorname{ch}_i(L_j)\quad \forall L_j\in H^2(X)\\
\operatorname{ch}_i(\boldsymbol{\operatorname{p}})&\mapsto 2\operatorname{ch}_i(\boldsymbol{\operatorname{p}}).
\end{align}

Now consider the following diagram:
\[\begin{tikzcd}
\mathbb{D}^X \arrow{r}{\Sigma^\ast} \arrow[swap]{d}{\phi} & \mathbb{D}^X \arrow{d}{\phi} \\%
\mathbb{D}^X \arrow{r}{g}& \mathbb{D}^X
\end{tikzcd},\]
where $\Sigma^\ast$ is defined in Lemma \ref{caldeslem}. 
By a simple calculation, one can find that the $g$ which makes the above diagram commute is the following map: \[g:\mathbb{D}^X\to\mathbb{D}^X:\operatorname{ch}_i(\gamma)\mapsto \operatorname{ch}_i(\iota^\ast \gamma ).\] Note that $g$ is the same as $\Sigma^\ast$, therefore the diagram just means $\phi$ and $\Sigma^\ast$ commute.

Lemma \ref{deslem} implies $g$ descends to a map $\tilde{g}$ on the level of cohomology, since the Markman operator $\phi$ satisfies property \eqref{desCond} with respect to Hilbert scheme of points by Lemma \ref{claim}.
By Lemma \ref{claim}, $\phi$ also descends to a Markman operator which is a cohomology ring isomorphism, which we also denote by $\phi$. Also note that all the geometric realization maps from $\mathbb{D}^X$ to cohomology rings are surjections as mentioned in the proof of the Lemma \ref{deslem}.
Therefore we have the following commutative diagram:
\[\begin{tikzcd}
H^\ast(M(v),\C) \arrow{r}{\sigma^\ast} \arrow[swap]{d}{\phi} & H^\ast(M(v),\C) \arrow{d}{\phi} \\%
H^\ast(X^{[N]},\C) \arrow{r}{\Tilde{g}}& H^\ast(X^{[N]},\C).
\end{tikzcd}\]
Since the geometric realization map in Lemma \ref{deslem} uses the sheaf $\mathcal{K}$, we have $\tilde{g}(\operatorname{ch}_i^\mathcal{K}(\gamma))=\operatorname{ch}_i^\mathcal{K}(\iota^\ast \gamma )$. 

However, for $X^{[N]}$ we have $\mathcal{K}=\mathcal{G}$, therefore we also have: $\tilde{g}(\operatorname{ch}_i^\mathcal{G}(\gamma))=\operatorname{ch}_i^\mathcal{G}(\iota^\ast \gamma )$. Apply the Lemma \ref{calofactioniota} and restrict $\iota^\ast$ to Hilbert scheme of points one get that $\iota^\ast$ act in the same way as $\tilde{g}$ on descendents. Therefore $\tilde{g}$ has to be the same as the map $\iota^\ast$, i.e., $\tilde{g}=\iota^\ast$. 

The $N$ of $X^{[N]}$ satisfies $2N-2=(v,v)$. By the same calculation as in the Remark \ref{odddimrmk}, we have that $N$ is an even number.  
\end{proof}

\begin{lemma}
\label{lem:hilbert-fix-point}
Let $N$ be an even number, and let $\pi:X\to Y$ be the covering K3 surface of the Enriques surface.
    We have that the fixed loci $\operatorname{Fix}(\iota^\ast)$ of the involution $\iota^\ast: {X}^{[N]}\to X^{[N]}$ is isomorphic to $Y^{[\frac{N}{2}]}$.
\end{lemma}
\begin{proof}
Let $M_Y$ denote the moduli space of all $\mu$-stable sheaves on $Y$, similarly for $M_{X}$. Define $M_Y^0:=\{E| E\in M_{Y}, E\ncong E(K_Y)\}$.
In \cite{Kim_1998}, Kim studied the map $\pi^\ast: M_Y^0\to M_{X}$ induced by the pull back of $\pi$. In our case, $\pi^\ast$ it is well defined for $Y^{[N/2]}$ since $Y^{[N/2]} \subseteq M_Y^0$. Kim showed that $\pi^\ast$ is two-to-one with no branching and the image of $M_Y^0$ by $\pi^\ast$ is a Lagrangian subvariety in $M_{X}$ and is
equal to the fixed locus by involution $\iota$. The proof of (2) of the main Theorem in \cite{Kim_1998} shows that if $\pi^\ast E\cong \pi^\ast E'$ then we have $E\cong E'$. Therefore the two cover map over $\operatorname{Fix}$ is given by:
\[ Y^{[N/2]} \sqcup M_{Y}(1,\omega_Y,1-N/2))\to \operatorname{Fix}(\iota^\ast), \]
and each component is isomorphic to $\operatorname{Fix}(\iota^\ast)$. Therefore we have $\pi^\ast:Y^{[N/2]}\to \operatorname{Fix}(\iota^\ast)$ is an isomorphism. 
\end{proof}

We can now use the Lefschetz formula (Proposition \ref{lefschetz}) to prove the following proposition.

\begin{proposition}
\label{eulercharoffix}
    Consider the moduli space $M_{\mathcal{A}/Y}(v_{\mathcal{A}})$ of torsion free $\mathcal{A}$-modules of rank one with twisted Mukai vector $v_{\mathcal{A}}$. Assume $Y\in B_M$ with $M<-\frac{v_\mathcal{A}^2+5}{2}$, then $M_{\mathcal{A}/Y}(v_{\mathcal{A}})$ is only non empty if its dimension is even. If we assume $\operatorname{dim}(M_{\mathcal{A}/Y}(v_{\mathcal{A}}))=N$ with $N=v_{\mathcal{A}}^2+1$ an even number, then $M_{\mathcal{A}/Y}(v_{\mathcal{A}})$ is smooth and
    \begin{align}
        e(M_{\mathcal{A}/Y}(v_{\mathcal{A}}))=2e(Y^{[N/2]}).
    \end{align}
\end{proposition}

\begin{proof}
Theorem \ref{thm:reede} shows that if $Y$ lies on the restricted Beauville locus $B_M$ then $M_{\mathcal{A}/Y}(v_\mathcal{A})$ is an étale double cover of the fixed locus $\operatorname{Fix}(\sigma)$. As shown in Lemma \ref{2diag}, the dimension of $\operatorname{Fix}(\sigma)$ is an even number, which implies $M_{\mathcal{A}/Y}(v_{\mathcal{A}})$ is only non empty if its dimension is even within $B_N$. The étale double cover $\Pi:M_{\mathcal{A}/Y}(v_\mathcal{A})\to \operatorname{Fix}(\sigma)$ implies that this is a unramified topological 2-sheet covering space. Hence we have
\begin{align}
e(M_{\mathcal{A}/Y})=2e(\operatorname{Fix}(\sigma))=2e(Y^{[N/2]}),
\end{align}
where the last equality is obtained by applying the Proposition \ref{lefschetz} to Lemma \ref{2diag} and using Lemma \ref{lem:hilbert-fix-point} and also the fact $\operatorname{Fix}(\sigma)$ is a smooth Lagrangian subvariety. 
\end{proof}

Now we can prove the main theorem of this paper (Theorem \ref{maintheorem} in the introduction):
\begin{theorem}
    Let $Y$ be an Enriques surface and let $\mathcal{A}$ be an Azumaya algebra of degree $2$ on $Y$ that represents the unique nontrivial Brauer class. Fix a polarization $H$ of $Y$. Let $c_1\in H^2(Y,\Z)_{\mathrm{tf}}$ be a class modulo torsion. Let $\Ch_\mathcal{A}=(2,c_1,\Ch_2)$ be a twisted Chern character. Let $M = M^H_{\mathcal{A}/Y}(\Ch_\mathcal{A})$ be the moduli space of stable twisted sheaves with twisted Chern character $\Ch_\mathcal{A}$. Denote by $N$ the virtual dimension of $M$. Then we have an equality
    \begin{equation}
        e^{\vir}(M) = \begin{cases}
            0 & \text{ if $N$ is odd,}\\
            2\cdot e(Y^{[\frac{N}{2}]}) & \text{ if $N$ is even.}
        \end{cases}
    \end{equation}
\end{theorem}

\begin{proof}
    Instead of twisted Chern character, we can use $\mathcal{A}$-Mukai vector $v_\mathcal{A}:=\Ch_\mathcal{A}\cdot \sqrt{\operatorname{td}_Y}$, so that $M=M^H_{\mathcal{A}/Y}(v_\mathcal{A})$ for some $v_\mathcal{A}$. By Theorem \ref{thm:deform-to-beauville-locus}, there is a smooth family of Enriques surfaces $f:\mathcal{Y}\to S$ over smooth base $S$. And there is a point $q \in S$ such that $\mathcal{Y}_q \in B_m$ with $m < -2v_{\mathcal{A}}^2 - 8$, there is a point $p \in S$ such that $\mathcal{Y}_p \cong Y$. By Proposition \ref{prp:existence-mod-space-enriques} and Remark \ref{rmk:existence-mod-space-enriques}, (possibly after shrinking $S$) there is a relative moduli space $M^\mathcal{H}_{\mathcal{A}/\mathcal{Y}/S}(v_{\mathcal{A}})\to S$ carrying a relative obstruction theory (here we use $v_{\mathcal{A}}$ and $\mathcal{A}$ to denote the Mukai vector and Azumaya algebra over the family $\mathcal{Y}$ by abuse of notation).  
    We have:
    \begin{align} e^\vir(M^H_{\mathcal{A}/Y}(v_\mathcal{A}))&=e(M^{\mathcal{H}_q}_{\mathcal{A}_q/\mathcal{Y}_q}(v_{\mathcal{A}}|_q))\\
    &=\begin{cases}
            0 & \text{ if $N$ is odd,}\\
            2\cdot e(\mathcal{Y}_q^{[\frac{N}{2}]})=2\cdot e(Y^{[\frac{N}{2}]}) & \text{ if $N$ is even.}
        \end{cases}\\
    \end{align}
    
    In the RHS of the first line we used $e(\bullet)$ since the moduli space is smooth within the Beauville locus $B_M$ therefore $e^\vir = e$, the last line used the Proposition \ref{eulercharoffix}. The fact $ e(\mathcal{Y}_q^{[\frac{N}{2}]})= e(Y^{[\frac{N}{2}]})$ can be seen from Göttsche's formula \cite{Gottsche1990} which states as following:
    \begin{align}
        \sum_{n=0}^{\infty} e\left(S^{[n]}\right) t^n=\prod_{m=1}^{\infty}\left(1-t^m\right)^{-e(S)}, \forall \text{ smooth projective surface $S$ over }\C.
    \end{align}
\end{proof}

\begin{remark}
    Our method can also be applied to untwisted sheaves. The proof is precisely the same, except that we need a different analysis to show that we can deform to a suitable situation. An Enriques surface is called nodal if there exists a smooth rational curve $R$. Otherwise, it is called unnodal. In the 10 dimensional moduli space of Enriques surfaces, a generic one is unnodal, the nodal ones form a 9 dimensional subvariety \cite{Kim_1998}. Therefore, we can always deform to such an unnodal surface. If $Y$ is unnodal and $v$ is a primitive Mukai vector, then the moduli space $M_Y(v)$ is always smooth of the expected dimension, see~\cite[pg. 7]{Nuer2015}. In this case, the main result of Kim~\cite{Kim_1998} shows that the analogue of Thm.~\ref{thm:reede} holds, i.e., the moduli of sheaves on the Enriques surface is a 2-to-1 \'etale cover of the fixed locus of $\iota^*$. From here the same application of the Lefschetz formula~\ref{lefschetz} gives the result. 
\end{remark}

\printbibliography
\end{document}